\documentclass{amsart}
\usepackage{amssymb,amsmath}
\usepackage{mathrsfs}
\usepackage{color}
\usepackage{pdfsync}

\newtheorem{theorem}{Theorem}[section]
\newtheorem{lemma}[theorem]{Lemma}
\newtheorem{proposition}[theorem]{Proposition}

\newtheorem{remark}[theorem]{Remark}
\numberwithin{equation}{section}

\newcommand{\NN}{\mathbb{N}}
\newcommand\RR{{{\mathbb R}}}

\newcommand\SSS{{\mathbb S}}
\newcommand\CC{{\mathbb C}}

\newcommand\cA{{\mathcal A}}

\newcommand\cF{{\mathcal F}}
\newcommand\cL{{\mathscr L}}

\def\la{\langle}
\def\ra{\rangle}

\begin{document}

\title[Analytic smoothing for Landau equation]
{Analytic smoothing effect\\ for the nonlinear
Landau equation\\
of Maxwellian molecules}
\author{Y. Morimoto \& C.-J. Xu}
\date{\today}
\address{\noindent \textsc{Yoshinori Morimoto, Graduate School of Human and Environmental Studies,
Kyoto University, Kyoto 606-8501, Japan}}
\email{morimoto.yoshinori.74r@st.kyoto-u.ac.jp }
\address{\noindent \textsc{Chao-Jiang Xu,
Department of Mathematics, Nanjing University of Aeronautics and Astronautics, Nanjing 211106,
P. R. China
\\
and\\
Universit\'e de Rouen-Normandie, CNRS UMR 6085, Laboratoire de Math\'ematiques, 76801 Saint-Etienne du Rouvray, France
}}
\email{xuchaojiang@nuaa.edu.cn}



\begin{abstract}
We consider the Cauchy problem of the nonlinear  Landau equation of Maxwellian molecules,
under the perturbation frame work to global equilibrium. We show that
if $H^r_x(L^2_v), r >3/2$ norm of the initial perturbation is small enough, then the Cauchy problem of the nonlinear Landau equation admits a unique global solution which becomes analytic with respect
to both position $x$ and velocity $v$ variables for any time $t>0$.
This is the first result of analytic smoothing effect for the spatially inhomogeneous nonlinear kinetic  equation. The method used here is microlocal analysis and energy estimates. The key point is adopting a time integral weight associated with the kinetic transport operator.
\end{abstract}

\maketitle

\section{Introduction}\label{S1}

We consider the Cauchy problem for the spatially inhomogeneous Landau equation, a kinetic model from plasma physics that describes
the evolution of a particle density $f(t,x,v) \ge 0$ with position $x \in \RR^3$  and velocity $v \in \RR^3$ at time  $t$.
It reads
\begin{equation}\label{landau-1}
\begin{cases}
\partial_tf+v\cdot\nabla_{x}f=Q_L(f,f),\\
f|_{t=0}=f_0,
\end{cases}
\end{equation}
where the term $Q_L(f,f)$ corresponds to the Landau collision operator associated to the bilinear operator
\[
Q_L(g, f)=\nabla_v \cdot \Big(\int_{\RR^3}a(v-v_*)\big(g(v_*)(\nabla_v f)(v)-(\nabla_v g)(v_*)f(v)\big)d v_*\Big).
\]
Here $a=(a_{i,j})_{1 \leq i,j \leq 3}$ stands for the non-negative symmetric matrix
\begin{equation}\label{landau_collision1}
a(v)=|v|^{\gamma}\bigl (|v|^2{\bf I} -v\otimes v\bigr) = |v|^{\gamma+2}\mathbb P_{v^{\perp}}
\in M_3(\RR),  \quad -3 \le \gamma<+\infty,
\end{equation}
where $\mathbb P_{v^{\perp}}$ is the orthogonal projection on $v^{\perp}.$

The Landau equation with $\gamma =-3$ was first proposed in 1936 \cite{landau} by
the Russian theoretical physicist Lev Davidovitch Landau, as a transport equation
for a system of charged particles,
where the long range of the Coulomb interactions makes it impossible
to use the normal Boltzmann equation.
It soon became (in combination with the Vlasov equation) the most important mathematical kinetic model in the theory of collisional plasma.

The ``generalized'' Landau equation with $\gamma >-3$ was independently introduced by several authors
( see e.g. \cite{Boby0,  D1992}). It plays a role of a model of the Boltzmann equation for various interactions including inverse
power law potential $\rho^{-n+1}, n >2$. This equation can be obtained as a limit of the Boltzmann equation when grazing collisions prevails
(see \cite{villani1, villani2}  for a detailed study of the limiting process and further references on the subjects).
Though the Landau equation has no relation to physics in the non-Coulomb case, from various mathematical points of view,  it
 has been intensively studied by mathematicians in last two decades,
because it is a simple approximation of the non-cutoff Boltzmann equation (see, e.g.\cite{ADVW, LMPX2}).
We refer the reader to the surveys  
\cite{pl2, villani2, bgp} and recent papers \cite{BPS, BGP, CA-MIS},
as well as to the references therein for matters related to the derivation and basic results for that equation.

In this article, we focus our attention on the analytic smoothing effect of a solution for the Cauchy problem \eqref{landau-1}.
More specifically, we study the Landau equation with Maxwellian molecules in {\it a close to equilibrium} framework.
The Maxwellian molecules 
corresponds to the case when the parameter $\gamma=0$ in the cross section \eqref{landau_collision1}.
We consider the fluctuation
$$
f=\mu+\sqrt{\mu}g,
$$
around the Maxwellian equilibrium distribution
\begin{equation}\label{maxwe}
\mu(v)=(2\pi)^{-\frac{3}{2}}e^{-\frac{|v|^2}{2}}.
\end{equation}
This distribution is a stationary solution for the Landau equation since it only depends on the velocity variable $v$ and
the fact that $Q_L(\mu,\mu)=0$. We consider the linearized Landau operator around this equilibrium distribution given by
\begin{equation*}
\cL g=-\mu^{-1/2}Q_L(\mu,\mu^{1/2}g)-\mu^{-1/2}Q_L(\mu^{1/2}g,\mu).
\end{equation*}
The Cauchy problem for the Landau equation \eqref{landau-1} is then reduced to the one for the fluctuation
\begin{equation}\label{landau-2}
\begin{cases}
\partial_t g+v\cdot\nabla_{x} g+\cL g=\Gamma (g, g),\\
g|_{t=0}=g_0,
\end{cases}
\end{equation}
with
\begin{equation*}
\Gamma(g, f)=\mu^{-1/2}Q_L(\sqrt{\mu}g,\sqrt{\mu}f).
\end{equation*}

To state our main result, we define
the Sobolev space $H^{r}_x(L^2_v)$ for $r\geq 0$  by
\begin{equation*}
H^{r}_x(L^2_v)=\bigl\{u \in L^2(\RR_{x,v}^{6}) ;\,\,\, \langle D_x \rangle^{r} u \in L^2(\RR_{x,v}^{6})\bigr\},
\end{equation*}
with $\langle\, D_x \rangle=\sqrt{1-\triangle_x}$ and $D_x=-i \partial_x$. Let $\cA(\RR^n)$ denote the analytic function space on $\RR^n$. 
\bigskip
\begin{theorem}\label{theorem-1}
Let $r>3/2$. There exists a small constant $\epsilon_0>0$ such that for all $g_0 \in H^{r}_x(L^2_v)$ satisfying
$$
\|g_0\|_{H^{r}_x(L^2_v)} \leq \epsilon_0,
$$
the Cauchy problem  $\eqref{landau-2}$ admits a unique global solution such that
$$
g(t)\in \cA (\RR^6_{x, v}),\quad \forall t>0\,.
$$
Furthermore, there exists a  $c_0>0$ such that,
\begin{equation}\label{analytic-smooth}
e^{c_0\{\tilde t^2 (-\Delta_x)^{1/2}+\tilde t (-\Delta_v)^{1/2}\} }g(t)\in L^{\infty}\bigl([0,+\infty),  H^{r}_x(L^2_v)\bigr),
\end{equation}
where $\tilde t=\min\{1, t\}$ for $t\geq 0$.
\end{theorem}

This article concerns the existence of a solution
to the Cauchy problem for the Landau equation, 
as well as the smoothing properties of that solution, which is
a topic studied in many previous works, as stated above.
Among them we refer \cite{villani1, DV1, DV2, D2015} concerning the existence result,
\cite{MX, MPX} about higher regularity such as (ultra-)analyticity, in the spatially homogeneous case, while
in the spatially inhomogeneous case, \cite{Alexandre-Villani} concerning renormalized solution with defect measure,
\cite{Guo, SG, CA-TR-Wu, CA-MIS, CLXX, DLSS}
in a close to equilibrium setting, and recent regularity results by
\cite{CSS, GIMV, Hen-Sn}  under boundedness conditions on the mass, energy, entropy densities.
Also we want to mention the related works on the non cut-off Boltzmann equation,
e.g., the papers by
\cite{ADVW, MUXY-DCDS, MU,  GLX,
AMUXY1,AMUXY2,AMUXY3, LMPX4, MWY, BHRV}.

The rest of paper is arranged as follows: In Section \ref{S2} we recall the exact expression of the linear operator given
in \cite{LMPX2}, introduce a version of the exponential weight used in the previous work \cite{MX}, and show the estimate for
the linear term with the exponential weight.
In Section \ref{S-3} we give an  explicit form of the nonlinear term by means of creation and annihilation operators and
spherical derivatives.  In Section \ref{S-4}, we use this to
show a trilinear estimate for the inner product of the nonlinear term and test function with exponential weight.
Section 5 is devoted to complete the proof of the main theorem by constructing the time local solution with analytic smoothing
property and by combining this and the known global existence result.

\section{Fourier analysis of linear Landau operator}\label{S2}

With $\mu$ defined in \eqref{maxwe},
the linearized Landau operator $\mathscr{L}$ is defined by
$$
\cL g=-\mu^{-1/2}Q_L(\mu,\mu^{1/2}g)-\mu^{-1/2}Q_L(\mu^{1/2}g,\mu),
$$
is an unbounded symmetric operator on $L^2(\RR^3_{v})$. In the case of Maxwellian molecules, the linearized Landau operator may be computed explicitly
(see e.g. \cite{LMPX2}, Proposition~1), and we have
$$
\cL =\cL_1+\cL_2,
$$
with
$$
\cL_{1}=2 \Big(-\Delta_v+\frac{|v|^2}{4}-\frac{3}{2}\Big)-\Delta_{\SSS^{2}},
$$
\begin{align*}
\cL_{2}&=\Big[\Delta_{\SSS^{2}}- 2\Big(-\Delta_v+\frac{|v|^2}{4}-\frac{3}{2}\Big)\Big]\mathbb{P}_1 \notag \\
&\quad +\Big[-\Delta_{\SSS^{2}}- 2\Big(-\Delta_v+\frac{|v|^2}{4}-\frac{3}{2}\Big)\Big]\mathbb{P}_2
\end{align*}
where
\begin{equation*}
\Delta_{\SSS^{2}}=\frac{1}{2}\sum_{\substack{1 \leq j,k \leq 3 \\ j \neq k}}L_{k,j}^2, \quad L_{k,j}=v_j \partial_{v_k}-v_k \partial_{v_j},
\end{equation*}
stands for the Laplace-Beltrami operator on the unit sphere $\SSS^{2}$, and
$\mathbb{P}_{k}$ denotes the orthogonal projection onto the Hermite basis $\mathcal{E}_k =
\mbox{Span}\{\Phi_\alpha\}_{\alpha \in \NN^3, |\alpha|=k}$. Here
$\Phi_{0}(v)=\mu^{1/2}(v)$,
\begin{equation*}
\Phi_{\alpha}=\frac{1}{\sqrt{\alpha!}}a_{+,1}^{\alpha_{1}}a_{+,2}^{\alpha_{2}} a_{+,3}^{\alpha_{3}} \Phi_{0}, \enskip \,\alpha=(\alpha_1,\alpha_2,\alpha_3)\in \NN^3, \ \alpha!=\alpha_1!\alpha_2!\alpha_3!,
\end{equation*}
with
\begin{equation*}
a_{\pm,j}=\frac{ v_{j}}2\mp\frac{\partial}{\partial v_{j}}, \quad 1 \leq j \leq 3.
\end{equation*}
We have then,
\begin{align*}
&(\cL_{1}g , g )_{L^2(\RR_v^3)} = 2 \sum_{j=1}^3 \left(\|\partial_{v_j} g\|^2_{L^2(\RR_v^3)}+ \frac{1}{4}
\|v_j g \|^2_{L^2(\RR_v^3)} \right)\\
& \qquad \qquad \qquad \quad + \frac{1}{2}\sum_{\substack{1 \leq j,k \leq 3 \\ j \neq k}}\|L_{k,j}g \|^2_{L^2(\RR_v^3)} -3\|g\|^2_{L^2(\RR_v^3)}.
\end{align*}
The operators $\cL_2$ is bounded in $L^2(\RR_v^3)$.   Putting
\begin{align*}
||| g|||^2_v=2\sum_{j=1}^3 \left(\|\partial_{v_j} g\|^2_{L^2(\RR_v^3)}+ \frac{1}{4}
\|v_j g \|^2_{L^2(\RR_v^3)} \right)+ \frac{1}{2}\sum_{\substack{1 \leq j,k \leq 3 \\ j \neq k}}\|L_{k,j}g \|^2_{L^2(\RR_v^3)}\\
||| g|||^2_{r, 0}=2\sum_{j=1}^3 \left(\|\partial_{v_j} g\|^2_{H^r_x(L^2_v)}+ \frac{1}{4}
\|v_j g \|^2_{H^r_x(L^2_v)} \right)+ \frac{1}{2}\sum_{\substack{1 \leq j,k \leq 3 \\ j \neq k}}\|L_{k,j}g \|^2_{H^r_x(L^2_v)},
\end{align*}
we have the following coercive estimates{;}  there exists $C>0$ such that for all  $g \in \mathscr{S}(\RR^3), $
\begin{equation*}
 ||| g|||^2_v
\leq (\cL g, g)_{L^2(\RR_v^3)}+C\|g\|_{L^2(\RR_v^3)}^2 \, .
\end{equation*}

\noindent 
{\bf Ultra-analytic smoothing effect of Kolmogorov  equation}.
We recall the Cauchy problem of Kolmogorov equation
$$
\begin{cases}
\partial_t f +  v \cdot \nabla_x f -\Delta_v f=0\,,\\
f(0,x, v)= f_0(x, v) \in L^2(\RR_{x,v}^{2n}).
\end{cases}
$$
Using the Fourier transform ( $(x,v) \leftrightarrow (\eta,\xi)$ ) we have
$$
\hat{f}(t, \eta, \xi)=e^{-\int^t_0|\xi+\rho \eta|^{2}
d\rho }\hat{f}_0(\eta, \xi+t\eta)
$$
because $(\partial_t - \eta \cdot \nabla_\xi + |\xi|^{2})\hat f =0$. Using the following Ukai inequality  \eqref{4.1} with $\alpha=2$,
\begin{equation*}
\exists c >0\,,\enskip
c \,(t|\xi|^{2}+t^{3}|\eta|^{2})\leq
\int^t_0|\xi+ \rho \eta|^{2}d\rho.
\end{equation*}
Then we have
$$
e^{c(-t\Delta_v -t^3\Delta_x)}
f(t, \,\cdot,\,\cdot\,)\in
L^2(\RR_{x,v}^{2n})\,,
$$
which shows the ultra-analytic smoothing effect of $(x, v)$ variables for any $t>0$.
\begin{lemma}\label{lemm4.1}
For any $\alpha >0$,  there exists $c_{\alpha} >0$ such that
\begin{equation}\label{4.1}
\int^t_0 (1+| \xi +\rho \eta |^2) ^{\alpha/2} d\rho \geq c_{\alpha} t\left \{1+ |\xi|^2 + t^2 |\eta|^2\right \}^{\alpha/2},
\end{equation}
for all $t>0$.
\end{lemma}
\begin{remark}
The following simple proof is due to Seiji Ukai. On the other hand, there exists $C_\alpha>0$ such that
\begin{equation}\label{4.1++}
\int^t_0 (1+| \xi +\rho \eta |^2) ^{\alpha/2} d\rho \leq C_{\alpha} t\left \{1+ |\xi|^2 + t^2 |\eta|^2\right \}^{\alpha/2},\quad \forall t>0.
\end{equation}
\end{remark}
\begin{proof}  Put $\rho = t \tau$, $\tilde \eta = -t  \eta$, then
the estimate \eqref{4.1} is equivalent to
$$
\int^1_0\la \xi- \tau \tilde \eta\ra ^\alpha d\tau \geq c_\alpha
\left \{1+ |\xi|^2 + |\tilde \eta|^2\right \}^{\alpha/2},
$$
with the notation $\langle\, \cdot\, \rangle=\sqrt{1+|\cdot|^2}$.
Since the case $\tilde \eta =0$ is trivial, we assume
$\tilde \eta \ne 0$.
Notice
\begin{align*}
2 \int^1_0\la \xi- \tau \tilde \eta\ra^\alpha d\tau \geq 1 +
\int^1_0|\xi- \tau \tilde \eta|^\alpha d\tau.
\end{align*}
If $|\xi| \geq |\tilde \eta|$
\begin{align*}
\int^1_0| \xi- \tau \tilde \eta| ^\alpha d\tau& \geq
 |\xi|^\alpha
\int^{1}_0\Big( 1 - \tau \frac{|\tilde \eta|}{|\xi |}\Big)^\alpha d\tau \geq
 |\xi|^\alpha
\int^{1}_0\big( 1 -  \tau\big)^\alpha d\tau \\
&= \frac{|\xi|^\alpha}{\alpha+1} \geq
\frac{1}{2^\alpha (\alpha+1)}(|\xi| ^2+ |\tilde
\eta|^2 )^{\alpha/2}.
\end{align*}
If $|\xi| < |\tilde \eta|$
\begin{align*}
\int^1_0|\xi- \tau \tilde \eta|^\alpha d\tau &
\geq |\tilde \eta|^\alpha
\int^{1}_{0} \left| \tau - \frac{|\xi|}{|\tilde \eta|} \right|^\alpha d \tau \\
&=
|\tilde \eta|^\alpha \left \{
\int_{0}^{|\xi|/|\tilde \eta|}
\Big( \frac{|\xi|}{|\tilde \eta|}- \tau \Big)^\alpha d \tau
+ \int_{|\xi|/|\tilde \eta|}^1
\Big( \tau - \frac{|\xi|}{|\tilde \eta|} \Big)^\alpha d \tau \right\}
\\
&\geq  \frac{|\tilde \eta|^\alpha}{\alpha+1}  \min_{0\leq \theta \leq 1}
( \theta^{\alpha+1} + (1-\theta )^{\alpha+1})
= \frac{|\tilde \eta|^\alpha}{2^\alpha (\alpha+1)}\\
& \geq
\frac{1}{2^{2\alpha}(\alpha+1)}(|\xi| ^2+ |\tilde
\eta|^2 )^{\alpha/2}.
\end{align*}
We finally get (\ref{4.1}).
\end{proof}

We set
$$
\Psi(t, \eta, \xi)= c_0 \int^t_0\la \xi+\rho\eta\ra d\rho = c_0 \int^t_0 \la \xi+(t-\rho) \eta \ra d\rho,
$$
for a sufficiently small $c_0>0$ which will be chosen later on. We have that
$$
(\partial_t - \eta \cdot \nabla_\xi)\Psi = c_0 \la \xi \ra.
$$
Then
we can use (\ref{4.1}) with $\alpha=1$ to estimate $\Psi$.

To study the Gevery (and analytic) regularity of kinetic equations, exponential
type weights were used in \cite{MUXY-DCDS,MX} (see also \cite{BHRV}). Now we set
$$
F_{\delta, \delta'}(t,\,\eta,\, \xi)= \frac{e^{\Psi}}{(1+\delta e^{\Psi})(1+
\delta' \Psi)^r }
$$
for $0< \delta\leq1$,  $r>3/2, 0<r \delta' \leq 1$. Without the confusion, we use the same notation $F_{\delta, \delta'}$ for the symbol $F_{\delta, \delta'}(t,\,\eta,\, \xi)$ and also the the pseudo-differential operator $F_{\delta, \delta'} (t, D_x, D_v)$.
If $A$ is a first
order differential operator of $(t,\eta,\xi)$ variables, then we have
\begin{equation}\label{4.3}
A F_{\delta, \delta'} =
 \left( \frac{1}{1+\delta e^{\Psi}}
 -\frac{r \delta'}{1+ \delta' \Psi}\right)
 (A \Psi) F_{\delta, \delta'},
\end{equation}
and
$$
\left| \frac{1}{1+\delta e^{\Psi}}
 -\frac{r \delta' }{1+ \delta' \Psi}\right|\leq {1, \, \mbox{since $0\le a, b \le 1$ implies $|a-b| \le 1$.}}
$$
We study now the apriori estimate of solution $g\in L^\infty((0, T), H^{r}_x(L^2_v))$
of the equation
\[
\partial_t g +   {v \cdot \nabla_x} g + \mathscr{L} g = \Gamma(g, g)\,.
\]
Since in this case,
$$
{v \cdot \nabla_x} g \in L^\infty((0, T), H^{r-1}_x(L^2_{-1}(\RR^3_v))),
$$
$$
\mathscr{L} g,\,\,\, \Gamma(g, g)\in L^\infty((0, T), H^{r}_x(H^{-2}_{-2}(\RR^3_v))),
$$
we take
$$
\tilde g=F_{\delta, \delta'} (t, D_x, D_v)\la \delta' v \ra^{-4} F_{\delta, \delta'}(t,D_x, D_v) g\in L^\infty((0, T), H^{2r}_{4}(\RR^6_{x, v}))
$$
as test function, where $H^{m}_{\ell}(\RR^6_{x, v})$ is the weighted Sobolev
space of order $\ell$ with respect to $v$ variable.
Taking the $H_x^{r}( L^2_v)$ inner product for $r>\frac 32$, we get,
\begin{align*}
&\Big((\partial_t  +   {v \cdot \nabla_x} ) g,
F_{\delta, \delta'} \la \delta' v \ra^{-4} F_{\delta, \delta'} g \Big)_{H_x^{r}( L^2_v)} + \Big(\mathscr{L} g, F_{\delta, \delta'} \la \delta' v \ra^{-4} F_{\delta, \delta'} g \Big)_{H_x^{r}( L^2_v)}\\
&\qquad\qquad\qquad\qquad\qquad= \Big( \Gamma(g, g),  F_{\delta, \delta'} \la \delta' v \ra^{-4} F_{\delta, \delta'}g \Big)_{H_x^{r}( L^2_v)}.
\end{align*}
For the first term, we have
\begin{proposition}\label{lemma2.1a}
There exist {$C_1, C_2>0$ independent of $\delta, \delta'$} such that, for $0<t\leq 1$,
\begin{align*}
&\Big((\partial_t  +   {v \cdot \nabla_x} ) g, F_{\delta, \delta'} \la \delta' v \ra^{-4} F_{\delta, \delta'} g \Big)_{H_x^{r}( L^2_v)}\\
&\ge \frac{1}{2}\frac{d}{dt}
\|\la \delta' v \ra^{-2} F_{\delta, \delta'} g\|^2_{H^r_x(L^2_v)}-
{c_0 C_1\|\la D_v\ra}  \la \delta' v \ra^{-2} F_{\delta, \delta'} g \|^2_{H_x^{r}( L^2_v)}\\
&\qquad \qquad \qquad \qquad \qquad \qquad - {C_2 }\|\la \delta' v \ra^{-2} F_{\delta, \delta'} g \|_{H_x^{r}( L^2_v)}\, .
\end{align*}	
\end{proposition}

\begin{proof} Using the Plancherel {formula}, we have
\begin{align*}
&	\Big((\partial_t  +   {v \cdot \nabla_x} ) g, F_{\delta, \delta'} \la \delta' v \ra^{-4} F_{\delta, \delta'} g \Big)_{H_x^{r} (L^2_v)}\\
&=	\frac{1}{(2\pi)^6}\Big(\la \delta' D_\xi \ra^{-2}F_{\delta, \delta'}(\partial_t   - \eta \cdot \nabla_\xi) \hat{g} ,  \la \eta \ra^{2 r} \la \delta' D_\xi \ra^{-2}F_{\delta, \delta'}  \hat{g}(t,\eta,\xi)\Big)_{L^2_{\eta, \xi}}\\
&=	\frac{1}{(2\pi)^6}\Big((\partial_t   - \eta \cdot \nabla_\xi) \la \delta' D_\xi \ra^{-2}F_{\delta, \delta'}\hat{g} ,  \la \eta \ra^{2 r} \la \delta' D_\xi \ra^{-2}F_{\delta, \delta'}  \hat{g}(t,\eta,\xi)\Big)_{L^2_{\eta, \xi}}\\
&\quad +\frac{1}{(2\pi)^6}	\Big([\la \delta' D_\xi \ra^{-2}F_{\delta, \delta'},\, (\partial_t   - \eta \cdot \nabla_\xi) ]\hat{g} ,  \la \eta \ra^{2 r} \la \delta' D_\xi \ra^{-2}F_{\delta, \delta'}  \hat{g}(t,\eta,\xi)\Big)_{L^2_{\eta, \xi}}\,,
\end{align*}
and
\begin{align*}
&\Big((\partial_t   - \eta \cdot \nabla_\xi) \la \delta' D_\xi \ra^{-2}F_{\delta, \delta'}\hat{g} ,  \la \eta \ra^{2 r} \la \delta' D_\xi \ra^{-2}F_{\delta, \delta'}  \hat{g}(t,\eta,\xi)\Big)_{L^2_{\eta, \xi}}\\
&=\frac{1}{2}\frac{d}{dt}\int_{\RR^6}
|\la \delta' D_\xi \ra^{-2}F_{\delta, \delta'}  \hat{g}(t,\eta,\xi)|^2 \la \eta \ra^{2 r} \frac{d \eta d\xi}{(2\pi)^6}\,. 
\end{align*}
We study now the commutators term.
Since $(\partial_t - \eta \cdot \nabla_\xi)\Psi = c_0 \la \xi \ra$,  we have
\[
{-[F_{\delta, \delta'},\, (\partial_t   - \eta \cdot \nabla_\xi) ]=(\partial_t - \eta \cdot \nabla_\xi)F_{\delta, \delta'} }=
 c_0 \la \xi \ra \left( \frac{1}{1+\delta e^{\Psi}}
 -\frac{r \delta'}{1+ \delta' \Psi}\right)
 F_{\delta, \delta'}\,,
\]
and
\begin{align*}
&[\la \delta' D_\xi \ra^{-2}F_{\delta, \delta'},\, (\partial_t   - \eta \cdot \nabla_\xi) ]
=\la \delta' D_\xi \ra^{-2}[F_{\delta, \delta'},\, (\partial_t   - \eta \cdot \nabla_\xi) ]\\
&
= -c_0\left(\la \delta' D_\xi \ra^{-2} \la \xi \ra \left( \frac{1}{1+\delta e^{\Psi}}
 -\frac{r \delta'}{1+ \delta' \Psi}\right) \la \delta' D_\xi \ra^{2}\right)\la \delta' D_\xi \ra^{-2} F_{\delta, \delta'}\,.
\end{align*}
Moreover, we have
\begin{align*}
&\la \delta' D_\xi \ra^{-2} \left( \frac{\la \xi \ra}{1+\delta e^{\Psi}}
 -\frac{r \delta'\la \xi \ra}{1+ \delta' \Psi}\right) \la \delta' D_\xi \ra^{2}=\la \xi \ra \left( \frac{1}{1+\delta e^{\Psi}}
 -\frac{r \delta'}{1+ \delta' \Psi}\right)
\\
&\qquad \qquad \qquad +\la \delta' D_\xi \ra^{-2} \left[\la \xi \ra \left( \frac{1}{1+\delta e^{\Psi}}
 -\frac{r \delta'}{1+ \delta' \Psi}\right), \la \delta' D_\xi \ra^{2}\right ] ,
\end{align*}
and
\begin{align*}
&\la \delta' D_\xi \ra^{-2} \left[\la \xi \ra \left( \frac{1}{1+\delta e^{\Psi}}
 -\frac{r \delta'}{1+ \delta' \Psi}\right), \la \delta' D_\xi \ra^{2}\right ] \\
 &=2\frac{( \delta')^2 D_\xi}{\la \delta' D_\xi \ra^{2}} {\cdot}\left(D_\xi\left( \frac{\la \xi \ra }{1+\delta e^{\Psi}}
 -\frac{r \delta'\la \xi \ra }{1+ \delta' \Psi}\right)\right) \\
&\quad -\la \delta' D_\xi \ra^{-2}( \delta')^2  \left(D^2_\xi\left( \frac{\la \xi \ra }{1+\delta e^{\Psi}}
 -\frac{r \delta'\la \xi \ra }{1+ \delta' \Psi}\right)\right) \,.
\end{align*}
Now using
\begin{equation}\label{derive-de-psi}
|\partial_{\xi_j}\Psi |= c_0 \left|\int_0^t \frac{\xi_j -\rho  \eta_j}{\la \xi - \rho \eta \ra }d\rho \right| \leq c_0 t,
\enskip |\partial^\alpha _{\xi}\Psi | \le C_\alpha c_0 t \enskip \mbox{for $|\alpha| \geq 2$},
\end{equation}
we can complete the proof of Proposition  \ref{lemma2.1a}.
\end{proof}

We study now {terms concerning linear operators.}
\begin{proposition}\label{prop2.3}
If $0<c_0 \ll 1$, then there exists a $C >0$ independent of $\delta, \delta' >0$ such
that for any $0<t \le 1$ we have
\begin{align*}
&\Big(\cL_{1} g,  F_{\delta, \delta'} \la \delta' v \ra^{-4} F_{\delta, \delta'}  g \Big)_{H_x^{r}( L^2_v)}\\
&\geq  \frac 1 2  ||| \la \delta' v \ra^{-2} F_{\delta, \delta'} g |||^2_{r, 0}-
C \|\la \delta' v \ra^{-2} F_{\delta,\delta'} g\|^2_{H_x^{r}( L^2_v)}\,.
\end{align*}
\end{proposition}

\begin{proof}
Note that
$$
\cL_{1}=2 \sum^3_{j=1}(D^2_{v_j}+\frac{v_j^2}{4})-3-\frac{1}{2}\sum_{\substack{1 \leq j,k \leq 3 \\ j \neq k}}L_{k,j}^2.
$$
Then we have firstly
\begin{align*}
&\Big(D_{v_j}^2 g,  F_{\delta, \delta'} \la \delta' v \ra^{-4} F_{\delta, \delta'}  g \Big)_{H_x^{r}( L^2_v)}
=\Big(D_{v_j}^2 F_{\delta, \delta'}g,   \la \delta' v \ra^{-4} F_{\delta, \delta'}  g \Big)_{H_x^{r}( L^2_v)}\\
&=\Big(\la \delta' v \ra^{-2}D_{v_j}^2\la \delta' v \ra^{2}\la \delta' v \ra^{-2} F_{\delta, \delta'}  g,   \la \delta' v \ra^{-2} F_{\delta, \delta'}  g \Big)_{H_x^{r}( L^2_v)}\\
&=\Big(D_{v_j}^2\la \delta' v \ra^{-2}  F_{\delta, \delta'}g,   \la \delta' v \ra^{-2} F_{\delta, \delta'}  g \Big)_{H_x^{r}( L^2_v)}\\
&\quad+\Big(\la \delta' v \ra^{-2}[D_{v_j}^2, \la \delta' v \ra^{2}]\la \delta' v \ra^{-2}  F_{\delta, \delta'}g,   \la \delta' v \ra^{-2} F_{\delta, \delta'}  g \Big)_{H_x^{r}( L^2_v)}.
\end{align*}
Since
$$
[D_{v_j}^2, \la \delta' v \ra^{2}]={ -2\delta'^2- i4 }\delta'^2v_j D_{v_j},
$$
we have
\begin{align*}
&\left|\Big(\la \delta' v \ra^{-2}[D_{v_j}^2, \la \delta' v \ra^{2}]\la \delta' v \ra^{-2}  F_{\delta, \delta'}g,   \la \delta' v \ra^{-2} F_{\delta, \delta'}  g \Big)_{H_x^{r}( L^2_v)}\right|\\
&\leq 2\delta'^2\|\la \delta' v \ra^{-2} F_{\delta, \delta'}  g \|^2_{H_x^{r}( L^2_v)}\\
&\qquad\qquad+
4\delta'\|D_{v_j}\la \delta' v \ra^{-2} F_{\delta, \delta'}  g \|_{H_x^{r}( L^2_v)}\|\la \delta' v \ra^{-2} F_{\delta, \delta'}  g \|_{H_x^{r}( L^2_v)}\\
&\leq 10\delta'^2\|\la \delta' v \ra^{-2} F_{\delta, \delta'}  g \|^2_{H_x^{r}( L^2_v)}+
\frac 12\|D_{v_j}\la \delta' v \ra^{-2} F_{\delta, \delta'}  g \|^2_{H_x^{r}( L^2_v)},
\end{align*}
which { gives}
\begin{align}\label{2.1a}
\Big(D_{v_j}^2 g,  F_{\delta, \delta'} \la \delta' v \ra^{-4} F_{\delta, \delta'}  g \Big)_{H_x^{r}( L^2_v)}&\geq
\frac 12\|D_{v_j}\la \delta' v \ra^{-2} F_{\delta, \delta'}  g \|^2_{H_x^{r}( L^2_v)}\\
&\quad-
10\delta'^2\|\la \delta' v \ra^{-2} F_{\delta, \delta'}  g \|^2_{H_x^{r}( L^2_v)}.\notag
\end{align}
Now, for the second terms, we write
\begin{align*}
\Big(v_j^2 g,  F_{\delta, \delta'} \la \delta' v \ra^{-4} F_{\delta, \delta'}  g \Big)_{H_x^{r}( L^2_v)}
&=\Big(v_j^2 \la \delta' v \ra^{-2} F_{\delta, \delta'}g,  \la \delta' v \ra^{-2} F_{\delta, \delta'}  g \Big)_{H_x^{r}( L^2_v)}\\
&+\Big( \la \delta' v \ra^{-2}[F_{\delta, \delta'},\, v_j^2]g,  \la \delta' v \ra^{-2} F_{\delta, \delta'}  g \Big)_{H_x^{r}( L^2_v)},
\end{align*}
and note that
\begin{align}\label{later-use}
[F_{\delta, \delta'},\, v_j^2]&=2 (D_{\xi_j} F_{\delta, \delta'})(t, D_x, D_v) v_j - (D_{\xi_j} ^2 F_{\delta, \delta'})(t, D_x, D_v)\\
& = 2 v_j (D_{\xi_j} F_{\delta, \delta'})(t, D_x, D_v) +(D_{\xi_j} ^2 F_{\delta, \delta'})(t, D_x, D_v)\,. \notag
\end{align}
{\color{black}
By using \eqref{4.3}, we have
\begin{align}\label{later-use-2}\notag
(D_{\xi_j} F_{\delta, \delta'})(t, D_x, D_v) &=
 \left( \frac{1}{1+\delta e^{\Psi}}
 -\frac{r \delta'}{1+ \delta' \Psi}\right)
 (D_{\xi_j} \Psi) F_{\delta, \delta'}\\
& := B_{j, \delta, \delta'}(t,D_x, D_v) F_{\delta, \delta'}
\end{align}
and moreover
\begin{align}\label{later-use-3}\notag
(D_{\xi_j}^2 F_{\delta, \delta'})(t, D_x, D_v)
&= \Big((D_{\xi_j} B_{j, \delta, \delta'} ) (t, D_x, D_v)+ B_{j, \delta, \delta'}(t, D_x, D_v)^2 \Big) F_{\delta, \delta'}\\
&:= \tilde B_{j, \delta, \delta'}(t,D_x, D_v) F_{\delta, \delta'}.
\end{align}
Consequently
\begin{align*}
&\left|\Big( \la \delta' v \ra^{-2}[F_{\delta, \delta'},\, v_j^2]g,  \la \delta' v \ra^{-2} F_{\delta, \delta'}  g \Big)_{H_x^{r}( L^2_v)}\right|\\
&\qquad \leq 2 \|\la \delta' v \ra^{-2}B_{j,\delta, \delta'}\la \delta' v \ra^{2}\la \delta' v \ra^{-2}F_{\delta, \delta'}g\|_{H_x^{r}( L^2_v)}
\|v_j \la \delta' v \ra^{-2} F_{\delta, \delta'}g\|_{H_x^{r}( L^2_v)}\\
 &\qquad \qquad +
\|\la \delta' v \ra^{-2}\tilde B_{j,\delta, \delta'}\la \delta' v \ra^{2}\la \delta' v \ra^{-2}F_{\delta, \delta'}g\|_{H_x^{r}( L^2_v)}
\|\la \delta' v \ra^{-2} F_{\delta, \delta'}g\|_{H_x^{r}( L^2_v)}.
\end{align*}
Since it follows from \eqref{derive-de-psi} that $\la \delta' v \ra^{-2} B_{\delta, \delta'}\la \delta' v \ra^{2}
= B_{\delta, \delta'} - \delta'^2 \la \delta' v \ra^{-2}[B_{\delta, \delta'}, |v|^2]$ is an $L^2(\RR^6_{x, v})$ bounded operator with
a constant factor $c_0t$ and the same fact holds for $\tilde B_{\delta, \delta'}$, we have that
for $0<t\le 1$,
\begin{align}\label{2.1b}
&\Big(v_j^2 g,  F_{\delta, \delta'} \la \delta' v \ra^{-4} F_{\delta, \delta'}  g \Big)_{H_x^{r}( L^2_v)}\\
& \ge \frac 12\|v_j \la \delta' v \ra^{-2} F_{\delta, \delta'} g\|^2_{H_x^{r} (L^2_v)} - C_1
\| \la \delta' v \ra^{-2} F_{\delta, \delta'} g\|^2_{H_x^{r}( L^2_v)}\,.\notag
\end{align}
}

Finally, for the last terms, let  $\hat{L}_{k,j} = -\xi_k \partial_{\xi_j} + \xi_j \partial_{\xi_k}$. Then we have
\begin{align}\label{ljk-derive}
\hat{L}_{k,j} F_{\delta, \delta'}(t,\eta, \xi) =
\left( \frac{1}{1+\delta e^{\Psi}}
 -\frac{r \delta'}{1+ \delta' \Psi}\right)
F_{\delta, \delta'}\left( -\xi_k \partial_{\xi_j} \Psi + \xi_j \partial_{\xi_k} \Psi \right),
\end{align}
and hence, in view of $r \delta' \le 1$,
\begin{align*}
&\Big(-L_{k,j}^2 g,  F_{\delta, \delta'} \la \delta' v \ra^{-4} F_{\delta, \delta'}  g \Big)_{H_x^{r}( L^2_v)}\\
&= \int_{\RR^6}
\Big\{ \Big(\hat{L}_{k,j} { -}  \left( \frac{1}{1+\delta e^{\Psi}}
 -\frac{r \delta'}{1+ \delta' \Psi}\right)
\big(-\xi_k\partial_{\xi_j}\Psi + \xi_j \partial_{\xi_k}
\Psi \big)\Big) F_{\delta, \delta'}  \hat{g}
\Big\} \la \eta \ra^{2 r} \\
& \times \overline{
\Big\{ \Big(\hat{L}_{k,j} {+}  \left( \frac{1}{1+\delta e^{\Psi}}
 -\frac{r \delta'}{1+ \delta' \Psi}\right)
\big(-\xi_k\partial_{\xi_j}\Psi + \xi_j \partial_{\xi_k}
\Psi \big)\Big) \la \delta' D_\xi \ra^{-4}F_{\delta, \delta'}  \hat{g}
\Big\}
}
\frac{d \eta d\xi}{(2\pi)^6}\\
& \geq \|L_{k,j} \la \delta' v \ra^{-2} F_{\delta, \delta'} g\|^2_{H_x^{r}( L^2_v)} \\
& \quad -c_0^2 t^2 C_2 \Big(
\| \partial_{v_j} \la \delta' v \ra^{-2} F_{\delta, \delta'} g\|^2_{H_x^{r}( L^2_v)} + \| \partial_{v_k} \la \delta' v \ra^{-2} F_{\delta, \delta'} g\|^2_{H_x^{r}( L^2_v)}\Big)\\
& \quad
-c_0 t C_3 \|L_{k,j} \la \delta' v \ra^{-2} F_{\delta, \delta'} g\|_{H_x^{r}( L^2_v)}\\
&\qquad\qquad\times \Big(
\| \partial_{v_j} \la \delta' v \ra^{-2} F_{\delta, \delta'} g\|_{H_x^{r} L^2_v} + \| \partial_{v_k} \la \delta' v \ra^{-2} F_{\delta, \delta'} g\|_{H_x^{r}( L^2_v)}\Big)\\
&\quad -c_0^2 t^2  C_4
\| \la \delta' v \ra^{-2} F_{\delta, \delta'} g\|^2_{H_x^{r}( L^2_v)}\,,
\end{align*}
where we have used $\left[\hat{L}_{k,j}, \la \delta'D_\xi \ra\right] =0$, because $\la \delta'D_\xi \ra$ is radial.

Therefore, for $0<t\leq 1 $ and $0 < c_0\ll 1$ we get
\begin{align}\label{2.1c}
&\Big(-L_{k,j}^2 g,  F_{\delta, \delta'} \la \delta' v \ra^{-4} F_{\delta, \delta'}  g \Big)_{H_x^{r}( L^2_v)}\notag\\
& \geq \frac 12 \|L_{k,j} \la \delta' v \ra^{-2} F_{\delta, \delta'} g\|^2_{H_x^{r}( L^2_v)}\\
&-\frac 18 \Big(
\| \partial_{v_j} \la \delta' v \ra^{-2} F_{\delta, \delta'} g\|^2_{H_x^{r}( L^2_v)} + \| \partial_{v_k} \la \delta' v \ra^{-2} F_{\delta, \delta'} g\|^2_{H_x^{r}( L^2_v)}\Big)\notag\\
&-C_5 \| \la \delta' v \ra^{-2} F_{\delta, \delta'} g\|^2_{H_x^{r}( L^2_v)}\,.\notag
\end{align}
Combing the estimates \eqref{2.1a},\eqref{2.1b} and \eqref{2.1c}, we finish the proof of Proposition \ref{prop2.3}.
\end{proof}

We recall  the notations: $\Phi_{0}(v)=\mu^{1/2}(v)$,
\begin{equation*}
\Phi_{\alpha}=\frac{1}{\sqrt{\alpha!}}a_{+,1}^{\alpha_{1}}a_{+,2}^{\alpha_{2}} a_{+,3}^{\alpha_{3}} \Phi_{0}, \,\alpha=(\alpha_1,\alpha_2,\alpha_3)\in \NN^3, \ \alpha!=\alpha_1!\alpha_2!\alpha_3!,
\end{equation*}
with
\begin{equation*}
a_{\pm,j}=\frac{ v_{j}}2\mp\frac{\partial}{\partial v_{j}}, \quad 1 \leq j \leq 3.
\end{equation*}
\begin{align*}
\Phi_{e_k}=v_k \Phi_0, \, \Phi_{2e_k}=\frac{1}{\sqrt{2}}(v_k^2-1) \Phi_0, \,
\Phi_{e_j+e_k}=v_j v_k \Phi_0 \, \, \, (j \ne k),
\end{align*}
where  $(e_1,e_2,e_3)$ stands for the canonical basis of $\RR^3$.
\begin{proposition}\label{prop2.4}
For $0<t\leq 1$, there exists $C_6>0$ independent of $0< \delta \le 1, 0< \delta' < r^{-1}$ and $0<c_0 \ll1$
such that we have
\begin{align*}
&\left|\Big(\cL_{2}g,  F_{\delta, \delta'} \la \delta' v \ra^{-4} F_{\delta, \delta'}  h \Big)_{H_x^{r}( L^2_v)}\right|
\\
&\qquad \leq C_6\|\la \delta' v \ra^{-2} F_{\delta, \delta'}   g \|_{H_x^{r}( L^2_v)}\|\la \delta' v \ra^{-2} F_{\delta, \delta'}   h \|_{H_x^{r}( L^2_v)}\,.
\end{align*}
\end{proposition}

\begin{proof}
First we recall
\begin{align*}
\cL_{2}g&=\Big[\Delta_{\SSS^{2}}- 2\Big(-\Delta_v+\frac{|v|^2}{4}-\frac{3}{2}\Big)\Big]\mathbb{P}_1 g \\
&\quad +\Big[-\Delta_{\SSS^{2}}- 2\Big(-\Delta_v+\frac{|v|^2}{4}-\frac{3}{2}\Big)\Big]\mathbb{P}_2 g.
\end{align*}
Notice
\[
\mathbb{P}_1 g = \sum_{k=1}^3 (g, \Phi_{e_k})_{L^2(\RR^3_v)} \Phi_{e_k}, \,\, \mathbb{P}_2 g =  \sum_{|\alpha|=2}(g, \Phi_{\alpha})_{L^2(\RR^3_v)} \Phi_{\alpha}\,.
\]
Then
\begin{align*}
\cL_{2}g&=\sum_{k=1}^3 (g, \Phi_{e_k})_{L^2(\RR^3_v)} \tilde{\Phi}_{e_k}+ \sum_{|\alpha|=2} (g, \Phi_{\alpha})_{L^2(\RR^3_v)} \tilde{\Phi}_{\alpha}
\end{align*}
with
\begin{align*}
\tilde{\Phi}_{e_k}&=  \Big[\Delta_{\SSS^{2}}- 2\Big(-\Delta_v+\frac{|v|^2}{4}-\frac{3}{2}\Big)\Big]\Phi_{e_k}=p_{e_k}(v)e^{-\frac{|v|^2}{4}}, \\
 \tilde{\Phi}_{2}&=\Big[-\Delta_{\SSS^{2}}- 2\Big(-\Delta_v+\frac{|v|^2}{4}-\frac{3}{2}\Big)\Big]\Phi_{\alpha}=p_\alpha (v)e^{-\frac{|v|^2}{4}},
\end{align*}
where $p_*(v)$ are the polynomial of $v$-variables of 3 or 4 degrees. We then study one {of terms, where $\tilde \Phi$ denotes} $\Phi_{e_k}, \Phi_\alpha$,
\begin{align*}
&\Big((g, \Phi)_{L^2(\RR^3_v)} \tilde{\Phi},  F_{\delta, \delta'} \la \delta' v \ra^{-4} F_{\delta, \delta'}  h \Big)_{H_x^{r}( L^2_v)}\\
&=\int_{\eta, \xi}(\hat{g}(t, \eta, \cdot), \cF_{v^*}(\Phi))_{L^2(\RR^3_{\xi^*})} \cF_{v}(\tilde{\Phi})(\xi)   \overline{F_{\delta, \delta'} \la \delta' D_\xi \ra^{-4} F_{\delta, \delta'} \hat{ h}(t, \eta, \xi) }\la\eta\ra^{2r}\,\,\frac{d\eta d\xi}{(2\pi)^6}\\
&=\int_{\eta, \xi, \xi^*}\hat{g}(t, \eta, \xi^*) {\overline{\cF_{v^*}(\Phi)(\xi^*)} }\cF_{v}(\tilde{\Phi})(\xi)   \overline{F_{\delta, \delta'} \la \delta' D_\xi \ra^{-4} F_{\delta, \delta'} \hat{ h}(t, \eta, \xi) }\la\eta\ra^{2r}\,\,\frac{d\eta d\xi d\xi^*}{(2\pi)^9}\\
&=\int_{\eta, \xi, \xi^*}\big(\la \delta' D_{\xi^*} \ra^{-2}F_{\delta, \delta'}(t, \eta, \xi^*)\hat{g}(t, \eta, \xi^*)\big)\, \overline{\big(\la \delta' D_{\xi^*} \ra^{2}\cF_{v^*}(\Phi)(\xi^*) \big)} \cF_{v}(\tilde{\Phi})(\xi)\\
&\qquad \times \frac{ F_{\delta, \delta'}(t, \eta, \xi)}{ F_{\delta, \delta'}(t, \eta, \xi^*)} \overline{\la \delta' D_\xi \ra^{-4} F_{\delta, \delta'}(t, \eta, \xi) \hat{ h}(t, \eta, \xi) }\la\eta\ra^{2r}\,\,\frac{d\eta d\xi d\xi^*}{(2\pi)^9}\\
&+2\delta'^2\int_{\eta, \xi, \xi^*}\big(\la \delta' D_{\xi^*} \ra^{-2}\, F_{\delta, \delta'}(t, \eta, \xi^*)\hat{g}(t, \eta, \xi^*)\big)\, \overline{\big( D_{\xi^*} \cF_{v^*}(\Phi)(\xi^*) \big)} \cF_{v}(\tilde{\Phi})(\xi)\\
&\qquad \times \left( D_{\xi^*}\frac{ F_{\delta, \delta'}(t, \eta, \xi)}{ F_{\delta, \delta'}(t, \eta, \xi^*)}\right) \overline{\la \delta' D_\xi \ra^{-4} F_{\delta, \delta'}(t, \eta, \xi) \hat{ h}(t, \eta, \xi) }\la\eta\ra^{2r}\,\,\frac{d\eta d\xi d\xi^*}{(2\pi)^9}\\
&+\delta'^2\int_{\eta, \xi, \xi^*}\big(\la \delta' D_{\xi^*} \ra^{-2}\, F_{\delta, \delta'}(t, \eta, \xi^*)\hat{g}(t, \eta, \xi^*)\big)\, \overline{\big( \cF_{v^*}(\Phi)(\xi^*) \big)} \cF_{v}(\tilde{\Phi})(\xi)\\
&\qquad \times \left( D_{\xi^*}^2\frac{ F_{\delta, \delta'}(t, \eta, \xi)}{ F_{\delta, \delta'}(t, \eta, \xi^*)}\right) \overline{\la \delta' D_\xi \ra^{-4} F_{\delta, \delta'}(t, \eta, \xi) \hat{ h}(t, \eta, \xi) }\la\eta\ra^{2r}\,\,\frac{d\eta d\xi d\xi^*}{(2\pi)^9}.
\end{align*}
On the other hand, we have
\begin{align*}
\frac{ F_{\delta, \delta'}(t, \eta, \xi)}{ F_{\delta, \delta'}(t, \eta, \xi^*)}=	&
e^{c_0\int^t_0(\la \xi+\rho\eta\ra-\la \xi^*+\rho\eta\ra )d\rho}
\frac{1+\delta e^{c_0\int^t_0 \la\xi^*+\rho\eta\ra d\rho} }{1+\delta e^{c_0\int^t_0
\la\xi+\rho\eta\ra d\rho}}\\
&\qquad\qquad\times \left(\frac{1+\delta' c_0\int^t_0\la\xi^*+\rho\eta\ra d\rho }{1+\delta' c_0\int^t_0\la\xi+\rho\eta\ra d\rho}\right)^r.
\end{align*}
{Since }
\begin{align}\label{add.1}
\la\xi+\rho\eta\ra&=|(1, \xi+\rho\eta)|\leq |(1, \xi^*+\rho\eta)|+ |(0, \xi-\xi^*)|\\
&\leq \la\xi^*+\rho\eta\ra +|\xi|+|\xi^*|, \,\,\forall \xi,\,\xi^*\in\RR^3, \notag
\end{align}
we have, by using \eqref{4.3} and \eqref{derive-de-psi}, for $0<t\leq 1$, $1\leq p\leq 2$,
$$
\left| D^p_{\xi^*}\frac{ F_{\delta, \delta'}(t, \eta, \xi)}{ F_{\delta, \delta'}(t, \eta, \xi^*)}\right|
\leq Ce^{2c_0 t (|\xi|+|\xi^*|)}(1+|\xi|^r+|\xi^*|^r).
$$
Because, for $0\leq p\leq 2$
$$
|\xi|^re^{2c_0t |\xi|} D^p_{\xi} \cF_{v}(\tilde{\Phi})(\xi) \in L^2(\RR^3_\xi),
$$
we proved that, for $0<t\leq 1$,
\begin{align*}
&\left|\Big((g, \Phi)_{L^2(\RR^3_v)} \tilde{\Phi},  F_{\delta, \delta'} \la \delta' v \ra^{-4} F_{\delta, \delta'}  h \Big)_{H_x^{r}( L^2_v)}\right|\\
&\leq C\int_{\eta} \|\la \delta' D_{\xi^*} \ra^{-2} F_{\delta, \delta'} \hat{ g}(t, \eta, \cdot) \|_{L^2_{\xi^*}}
\|\la \delta' D_\xi \ra^{-4} F_{\delta, \delta'} \hat{ h}(t, \eta, \cdot) \|_{L^2_\xi}\la\eta\ra^{2r}\,\,\frac{d\eta}{(2\pi)^9}\\
&\leq C\|\la \delta' v \ra^{-2} F_{\delta, \delta'}  g \|_{H_x^{r}( L^2_v)}\|\la \delta' v \ra^{-2} F_{\delta, \delta'}  h \|_{H_x^{r}( L^2_v)}.
\end{align*}
\end{proof}

\begin{remark}
Compared to \eqref{add.1}, there is no constant $C>0$ such that
\[
\la\xi+\rho\eta\ra^\alpha-\la\xi^*+\rho\eta\ra^\alpha\leq C(|\xi|^\alpha+|\xi^*|^\alpha),\quad \forall \xi, \,\xi^*\in \RR^3,
\]
if  $\alpha > 1$. By this reason we do not seek for the ultra-analytic smoothing effect in the present paper.
\end{remark}
In conclusion, for the linear operators, we get that there exists $C_7>1$ independent
of $0< \delta \le 1, 0< \delta' < r^{-1}$ and $0<c_0 \ll1$ such that, for $0<t\leq 1$
\begin{align}\label{2.18} \notag
&\Big((\partial_t  +   {v \cdot \nabla_x} +{ {\cL}}) g,  F_{\delta, \delta'} \la \delta' v \ra^{-4} F_{\delta, \delta'} g \Big)_{H_x^{r}( L^2_v)} \\
&\ge \frac{1}{2}\frac{d}{dt}
\|\la \delta' v \ra^{-2} F_{\delta, \delta'} g\|^2_{H^r_x(L^2_v)}+\frac 14||| \la \delta' v \ra^{-2} F_{\delta, \delta'} g |||^2_{r, 0}
\\
&\qquad\qquad\qquad -C_7\|\la \delta' v \ra^{-2} F_{\delta, \delta'} g \|^2_{H_x^{r}( L^2_v)}\, . \nonumber
\end{align}

\section{Decomposition of nonlinear operators}\label{S-3}

We compute now the nonlinear term $\Gamma(f, g)$,
\begin{proposition}\label{prop3.1}
\begin{equation}\label{equ3.1}
(\Gamma (f,  g),  h)_{L^2(\RR_v^3)}=D_1+D_2+D_3+D_4+D_5+D_6+D_7,
\end{equation}
with
\begin{align*}
D_1&=\sqrt{2}\sum_{\substack{1 \leq i,j \leq 3\\ i \neq j }}(f,\Phi_{2e_j})_{L^2(\RR_v^3)}(a_{+,i}g,a_{-,i}h)_{L^2(\RR_v^3)},\\
D_2&=-\sum_{\substack{1 \leq i,j \leq 3\\ i \neq j }}(f,\Phi_0)_{L^2(\RR_v^3)}(a_{-,i}g,a_{-,i}h)_{L^2(\RR_v^3)},\\
D_3&=- \ \sum_{\substack{1 \leq i,j \leq 3\\ i \neq j }}(f,\Phi_{e_i+e_j})_{L^2(\RR_v^3)} (a_{+,j}g,a_{-,i}h)_{L^2(\RR_v^3)},
\end{align*}
\begin{align*}
D_4&=\sum_{\substack{1 \leq i,j \leq 3\\ i \neq j }}(f,\Phi_{e_i})_{L^2(\RR_v^3)}  (g  ,a_{-,i}h)_{L^2(\RR_v^3)},\\
D_5&= - \frac{1}{2}\sum_{\substack{1 \leq i,j \leq 3\\ i \neq j }}(f,\Phi_0)_{L^2(\RR_v^3)}\big(L_{i,j}g, L_{i,j}h\big)_{L^2(\RR_v^3)},\\
D_6&=\sum_{\substack{1 \leq i,j \leq 3\\ i \neq j }}(f,\Phi_{e_j})_{L^2(\RR_v^3)}(L_{i,j}g, a_{-,i}h)_{L^2(\RR_v^3)},\\
D_7&=- \sum_{\substack{1 \leq i,j \leq 3\\ i \neq j }}(f,\Phi_{e_j})_{L^2(\RR_v^3)}(a_{+,i}g,L_{i,j}h)_{L^2(\RR_v^3)}.
\end{align*}
where the creation (resp. annihilation) operator $a_{+,j}$ (resp. $a_{-,j}$) is given by
\begin{equation*}
a_{\pm,j}=\frac{ v_{j}}2\mp\frac{\partial}{\partial  v_{j}}\quad
\mbox{and}\quad
L_{k,j}=v_j \partial_{v_k}-v_k \partial_{v_j}.
\end{equation*}
\end{proposition}
\begin{proof}
We begin by computing explicitly the bilinear term $\Gamma(g, f)$. { Notice}  that for all $f, g \in \mathscr{S}(\RR_v^3)$,
\begin{multline*}
\Gamma (f, g)=\mu^{-1/2}(v)\sum_{1 \leq k,j \leq 3} \partial_{v_k}\Big(\int_{\RR^3}a_{k,j}(v-v_*)\mu^{1/2}(v_*)f(v_*)\partial_{v_j}\big(\mu^{1/2}(v)g(v)\big)dv_*\Big)\\
-\mu^{-1/2}(v)\sum_{1 \leq k,j \leq 3} \partial_{v_k}\Big(\int_{\RR^3}a_{i,j}(v-v_*)  \partial_{v_j^*}\big(\mu^{1/2}(v_*)f(v_*)\big)\mu^{1/2}(v)g(v)dv_*\Big),
\end{multline*}
that is,
\begin{multline*}
\Gamma (f,g)
=\sum_{1 \leq k,j \leq 3} \Big(\partial_{v_k}-\frac{v_k}{2}\Big)\Big(\int_{\RR^3}a_{k,j}(v-v_*)\mu^{1/2}(v_*)f(v_*)\Big(\partial_jg(v)-\frac{v_j}{2}g(v)\Big)dv_*\Big)\\
-\sum_{1 \leq k,j \leq 3} \Big(\partial_{v_k}-\frac{v_k}{2}\Big)\Big(\int_{\RR^3}a_{k,j}(v-v_*)  \mu^{1/2}(v_*)\Big(\partial_jf(v_*)-\frac{v_j^*}{2}f(v_*)\Big)g(v)dv_*\Big).
\end{multline*}
It follows that for all $f, g, h \in \mathscr{S}(\RR_v^3)$,
\begin{align*}
\langle\Gamma (f,g), &h\rangle_{L^2(\RR^3_v)}\\
= & \ \sum_{1 \leq k,j \leq 3}\int_{\RR^6}a_{k,j}(v-v_*)\mu^{1/2}(v_*)f(v_*)\Big(\partial_jg(v)-\frac{v_j}{2}g(v)\Big) \\
&\qquad\qquad\times\overline{\Big(-\partial_{k}h(v)-\frac{v_k}{2}h(v)\Big)}dv_*dv\\
& - \sum_{1 \leq k,j \leq 3} \int_{\RR^6}a_{k,j}(v-v_*)  \mu^{1/2}(v_*)\Big(\partial_jf(v_*)-\frac{v_j^*}{2}f(v_*)\Big)g(v) \\
&\qquad\qquad\times \overline{\Big(-\partial_{k}h(v)-\frac{v_k}{2}h(v)\Big)}dv_*dv.
\end{align*}
Integrating by parts, we obtain that
\begin{align*}
\int_{\RR^{3}}a_{k,j}(v-v_*)  \mu^{1/2}(v_*)\Big(\partial_jf(v_*)-\frac{v_j^*}{2}f(v_*)\Big)dv_*\\
=-\int_{\RR^{3}}\big(\partial_{v_j^*}a_{k,j}(v-v_*)\big)  \mu^{1/2}(v_*)f(v_*)dv_*,
\end{align*}
and this implies that
\begin{multline*}
 \langle \Gamma(f,g), h \rangle_{L^2(\RR^3_v)}
=
\sum_{1 \leq k,j \leq 3}\int_{\RR^{3}}\mu(v_*)^{1/2}f(v_*)
\\
\times \Big\langle a_{k,j}(v-v_*)\big(\partial_jg(v)-\frac{v_j}{2}g(v)\big)+\partial_{v_j^*}\big(a_{k,j}(v-v_*)\big)g(v)  ,-\partial_{k}h(v)-\frac{v_k}{2}h(v)\Big\rangle_{L^2(\RR_v^3)}dv_*\, .
\end{multline*}
Since for the Maxweillian case,
$$
a_{k, j}(z)=\delta_{k j}|z|^2-z_kz_j,\quad 1\leq k, j\leq 3\, .
$$
We have that
\begin{align*}
&\langle\Gamma (f, g), h\rangle_{L^2(\RR^3_v)}\\
&= \sum_{\substack{1 \leq k,j \leq 3\\ k \neq j }}\int_{\RR^{3}}\mu(v_*)^{1/2}f(v_*) \\ &\times\Big\langle(v_j^2-2v_jv_j^*+(v_j^*)^2)\big(\partial_k g(v)-\frac{v_k}{2}g(v)\big),-\partial_{k}h(v)-\frac{v_k}{2}h(v)\Big\rangle_{L^2(\RR_v^3)}dv_*
\\
&\qquad +
\sum_{\substack{1 \leq k,j \leq 3\\ k \neq j }}\int_{\RR^{3}}\mu(v_*)^{1/2}f(v_*)
\\
&\times
\Big\langle-(v_k-v_k^*)(v_j-v_j^*)\big(\partial_jg(v)-\frac{v_j}{2}g(v)\big)+(v_k-v_k^*)g(v)  ,\\
&\qquad\qquad\qquad\qquad\qquad\qquad\qquad\qquad\qquad\qquad-\partial_{k}h(v)-\frac{v_k}{2}h(v)\Big\rangle_{L^2(\RR_v^3)}dv_*.
\end{align*}
We obtain that
\begin{equation*}
 \langle\Gamma(g,f), h\rangle_{L^2(\RR^3_v)}=A_0+A_1+A_2+A_3+A_4+A_5+A_6+A_7,
\end{equation*}
with
\begin{align*}
A_0&= -\sum_{\substack{1 \leq k,j \leq 3\\ k \neq j }}\int_{\RR^{3}}v_k^*v_j^*\mu(v_*)^{1/2}f(v_*) \\
&\qquad\qquad\times\Big \langle\partial_jg(v)-\frac{v_j}{2}g(v),-\partial_{k}h(v)-\frac{v_k}{2}h(v)\Big \rangle_{L^2(\RR_v^3)}dv_*,
\end{align*}
\begin{align*}
&A_1= \sum_{\substack{1 \leq k,j \leq 3\\ k \neq j }}\int_{\RR^{3}}(v_j^*)^2\mu(v_*)^{1/2}f(v_*)\\
 &\qquad\qquad\times \Big \langle\partial_k g(v)-\frac{v_k}{2}g(v),-\partial_{k}h(v)-\frac{v_k}{2}h(v)\Big \rangle_{L^2(\RR_v^3)}dv_*,
\\
&A_2=\sum_{\substack{1 \leq k,j \leq 3\\ k \neq j }}\int_{\RR^{3}}\mu(v_*)^{1/2}f(v_*) \Big \langle v_k g(v)  ,-\partial_{k}h(v)-\frac{v_k}{2}h(v)\Big \rangle _{L^2(\RR_v^3)}dv_*,
\\
&A_3=-\sum_{\substack{1 \leq k,j \leq 3\\ k \neq j }}\int_{\RR^{3}}v_k^*\mu(v_*)^{1/2}f(v_*) \Big \langle g(v)  ,-\partial_{k}h(v)-\frac{v_k}{2}h(v)\Big \rangle_{L^2(\RR_v^3)}dv_*,
\\
&A_4=
\frac{1}{4}\sum_{\substack{1 \leq k,j \leq 3\\ k \neq j }}\int_{\RR^{3}}\mu(v_*)^{1/2}f(v_*)
\\
&\times
\big \langle \big(-v_k^2v_j^2+v_k^2v_jv_j^*+v_kv_j^2v_k^*+v_k^2v_j^2-2v_k^2v_jv_j^*\big)g(v),h(v)\big \rangle_{L^2(\RR_v^3)}dv_*=0,
\end{align*}
since
$\sum_{\substack{1 \leq k,j \leq 3\\ k \neq j }}\big(v_kv_j^2v_k^*-v_k^2v_jv_j^*\big)=0.$
We have also
\begin{multline}\label{j14}
A_5= \sum_{\substack{1 \leq k,j \leq 3\\ k \neq j }}\int_{\RR^{3}}\mu(v_*)^{1/2}f(v_*)\\\times  \Big \langle (v_j^2-2v_jv_j^*)\partial_k g(v)+\big(-v_kv_j+v_kv_j^*+v_jv_k^*\big)\partial_jg(v),-\partial_{k}h(v)\Big \rangle_{L^2(\RR_v^3)}dv_*,
\end{multline}
and
\begin{multline*}
A_6=
-\frac{1}{2}\sum_{\substack{1 \leq k,j \leq 3\\ k \neq j }}\int_{\RR^{3}}\mu(v_*)^{1/2}f(v_*) \big \langle v_k\partial_jg(v),(-v_kv_j+v_kv_j^*+v_jv_k^*)h(v)\big \rangle_{L^2(\RR_v^3)}dv_*\\
+\frac{1}{2}\sum_{\substack{1 \leq k,j \leq 3\\ k \neq j }}\int_{\RR^{3}}\mu(v_*)^{1/2}f(v_*) \big \langle (-v_kv_j+v_kv_j^*+v_jv_k^*)g(v),v_j\partial_{k}h(v)\big \rangle_{L^2(\RR_v^3)}dv_*,
\end{multline*}
\begin{multline*}
A_7= -\frac{1}{2}\sum_{\substack{1 \leq k,j \leq 3\\ k \neq j }}\int_{\RR^{3}}\mu(v_*)^{1/2}f(v_*) \big \langle v_j\partial_k g(v),(v_kv_j-2v_kv_j^*)h(v)\big \rangle_{L^2(\RR_v^3)}dv_*\\
+\frac{1}{2}\sum_{\substack{1 \leq k,j \leq 3\\ k \neq j }}\int_{\RR^{3}}\mu(v_*)^{1/2}f(v_*) \big \langle (v_kv_j-2v_kv_j^*)g(v),v_j\partial_{k}h(v)\big \rangle _{L^2(\RR_v^3)}dv_*.
\end{multline*}
It follows from (\ref{j14}) that
\begin{multline*}
A_5= \sum_{\substack{1 \leq k,j \leq 3\\ k \neq j }}\int_{\RR^{3}}\mu(v_*)^{1/2}f(v_*) \big \langle v_j\partial_kg(v)-v_k\partial_jg(v),-v_j\partial_{k}h(v)\big \rangle_{L^2(\RR_v^3)}dv_*\\
+\sum_{\substack{1 \leq k,j \leq 3\\ k \neq j }}\int_{\RR^{3}}\mu(v_*)^{1/2}f(v_*) \big \langle -2v_jv_j^*\partial_kg(v)+\big(v_kv_j^*+v_jv_k^*\big)\partial_jg(v),-\partial_{k}h(v)\big \rangle_{L^2(\RR_v^3)}dv_*.
\end{multline*}
In the sequel, we shall use several times the (obvious) formula
$
\sum_{\substack{1 \leq k,j \leq 3\\ k \neq j }}\alpha_{k,j}=\sum_{\substack{1 \leq k,j \leq 3\\ k \neq j }}\alpha_{j,k}.
$
We notice that
\begin{align*}
&\sum_{\substack{1 \leq k,j \leq 3\\ k \neq j }}\big \langle v_j\partial_kg(v)-v_k\partial_jg(v),-v_j\partial_{k}h(v)\big \rangle _{L^2(\RR_v^3)}\\ \notag
= & \ \frac{1}{2}\sum_{\substack{1 \leq k,j \leq 3\\ k \neq j }}\big \langle v_j\partial_kg(v)-v_k\partial_jg(v),-v_j\partial_{k}h(v)\big \rangle _{L^2(\RR_v^3)}\\ \notag
 &  \qquad  + \frac{1}{2}\sum_{\substack{1 \leq k,j \leq 3\\ k \neq j }}\big \langle v_k\partial_jg(v)-v_j\partial_kg(v),-v_k\partial_{j}h(v)\big \rangle _{L^2(\RR_v^3)}\\ \notag
= & \ -\frac{1}{2}\sum_{\substack{1 \leq k,j \leq 3\\ k \neq j }}\big \langle v_j\partial_kg(v)-v_k\partial_jg(v),v_j\partial_{k}h(v)-v_k\partial_{j}h(v)\big \rangle _{L^2(\RR_v^3)},
\end{align*}
and
\begin{align*}
\sum_{\substack{1 \leq k,j \leq 3\\ k \neq j }}& \big\langle  -2v_jv_j^*\partial_k g(v)+\big( v_kv_j^*+v_jv_{{ k}}^*\big)\partial_jg(v),-\partial_{k}h(v)\big\rangle _{L^2(\RR_v^3)}\\ \notag
=& \  \sum_{\substack{1 \leq k,j \leq 3\\ k \neq j }} \big[\big\langle  v_j\partial_kg(v),v_j^*\partial_{k}h(v)\big\rangle _{L^2(\RR_v^3)}+\big\langle  v_j^*\partial_kg(v),v_j\partial_{k}h(v)\big\rangle  _{L^2(\RR_v^3)}\big]\\ \notag
& \hskip17pt  -\sum_{\substack{1 \leq k,j \leq 3\\ k \neq j }} \big\langle  v_k\partial_jg(v),v_j^*\partial_{k}h(v)\big\rangle _{L^2(\RR_v^3)}\\
&\qquad\qquad\qquad-\sum_{\substack{1 \leq k,j \leq 3\\ k \neq j }} \big\langle  v_j^*\partial_kg(v),v_k\partial_{j}h(v)\big\rangle _{L^2(\RR_v^3)}\\ \notag
=& \  \sum_{\substack{1 \leq k,j \leq 3\\ k \neq j }} \big[\big\langle  v_j\partial_kg(v)-v_k\partial_jg(v),v_j^*\partial_{k}h(v)\big\rangle _{L^2(\RR_v^3)}\\
&\qquad\qquad\qquad+\big\langle  v_j^*\partial_kg(v),v_j\partial_{k}h(v)-v_k\partial_{j}h(v)\big\rangle _{L^2(\RR_v^3)}\big].
\end{align*}
This implies that
\begin{align*}
A_5&=  -\frac{1}{2}\sum_{\substack{1 \leq k,j \leq 3\\ k \neq j }}\int_{\RR^{3}}\mu(v_*)^{1/2}f(v_*) \big\langle v_j\partial_kg(v)-v_k\partial_jg(v),\\
&\qquad\qquad\qquad\qquad\qquad\qquad\qquad\qquad v_j\partial_{k}h(v)-v_k\partial_{j}h(v)\big\rangle _{L^2(\RR_v^3)}dv_*\\ &
\hskip22pt +
\sum_{\substack{1 \leq k,j \leq 3\\ k \neq j }}\int_{\RR^{3}}v_j^*\mu(v_*)^{1/2}f(v_*) \Big[\big\langle  v_j\partial_kg(v)-v_k\partial_jg(v),\partial_{k}h(v)\big\rangle _{L^2(\RR_v^3)}
\\
&\qquad\qquad\qquad\qquad\qquad\qquad+\big\langle \partial_kg(v),v_j\partial_{k}h(v)-v_k\partial_{j}h(v)\big\rangle _{L^2(\RR_v^3)}\Big]dv_*.
\end{align*}
On the other hand, we may write that
\begin{align*}
A_6&+A_7=  \\ \notag
  &-\frac{1}{2}\sum_{\substack{1 \leq k,j \leq 3\\ k \neq j }}\int_{\RR^{3}}\mu(v_*)^{1/2}f(v_*) \big\langle  v_k\partial_jg(v),\\
  &\qquad\qquad\qquad\qquad\qquad(-v_kv_j+v_kv_j^*+v_jv_k^*+v_j-2v_jv_k^*)h(v)\big\rangle _{L^2(\RR_v^3)}dv_*\\ \notag
  &+\frac{1}{2}\sum_{\substack{1 \leq k,j \leq 3\\ k \neq j }}\int_{\RR^{3}}\mu(v_*)^{1/2}f(v_*) \big\langle  (-v_{{ k}}v_j+v_{{ k}}v_j^*+v_jv_{{ k}}^*+v_{{ k}}v_j-2v_{{ k}}v_j^*)g(v),\\
  &\qquad\qquad\qquad\qquad\qquad\qquad\qquad\qquad\qquad\qquad\qquad v_j\partial_{k}h(v)\big\rangle _{L^2(\RR_v^3)}dv_*.
\end{align*}
It follows that
\begin{multline*}
 A_6+A_7=  -\frac{1}{2}\sum_{\substack{1 \leq k,j \leq 3\\ k \neq j }}\int_{\RR^{3}}\mu(v_*)^{1/2}f(v_*) \big\langle  v_k\partial_jg(v),(v_kv_j^*-v_jv_k^*)h(v)\big\rangle _{L^2(\RR_v^3)}dv_*\\
  +\frac{1}{2}\sum_{\substack{1 \leq k,j \leq 3\\ k \neq j }}\int_{\RR^{3}}\mu(v_*)^{1/2}f(v_*) \big\langle  (v_jv_k^*-v_kv_j^*)g(v),v_j\partial_{k}h(v)\big\rangle _{L^2(\RR_v^3)}dv_*.
\end{multline*}
This implies that
\begin{multline*}
 A_6+A_7=  -\frac{1}{4}\sum_{\substack{1 \leq k,j \leq 3\\ k \neq j }}\int_{\RR^{3}}\mu(v_*)^{1/2}f(v_*) \big\langle  v_k\partial_jg(v)-v_j\partial_kg(v),\\
 \qquad\qquad\qquad\qquad\qquad\qquad\qquad\qquad(v_kv_j^*-v_jv_k^*)h(v)\big\rangle _{L^2(\RR_v^3)}dv_*\\
  +\frac{1}{4}\sum_{\substack{1 \leq k,j \leq 3\\ k \neq j }}\int_{\RR^{3}}\mu(v_*)^{1/2}f(v_*) \big\langle  (v_jv_k^*-v_kv_j^*)g(v),v_j\partial_{k}h(v)-v_k\partial_{j}h(v)\big\rangle _{L^2(\RR_v^3)}dv_*,
\end{multline*}
that is,
\begin{multline*}
 A_6+A_7=  \frac{1}{2}\sum_{\substack{1 \leq k,j \leq 3\\ k \neq j }}\int_{\RR^{3}}v_j^*\mu(v_*)^{1/2}f(v_*) \big\langle  v_j\partial_kg(v)-v_k\partial_jg(v),v_{{ k}}h(v)\big\rangle _{L^2(\RR_v^3)}dv_*\\
  +\frac{1}{2}\sum_{\substack{1 \leq k,j \leq 3\\ k \neq j }}\int_{\RR^{3}}v_k^*\mu(v_*)^{1/2}f(v_*) \big\langle  v_jg(v),v_j\partial_{k}h(v)-v_k\partial_{j}h(v)\big\rangle_{L^2(\RR_v^3)}dv_*.
\end{multline*}
We obtain that
\begin{equation*}
\langle \Gamma (f, g), h\rangle _{L^2(\RR_v^3)} =E_1+E_2+E_3+E_4+E_5+E_6+E_7,
\end{equation*}
with
\begin{align*}
E_1&=\sum_{\substack{1 \leq k,j \leq 3\\ k \neq j }}\big\langle f,v_j^2\mu^{1/2}\big\rangle _{L^2(\RR_v^3)} \Big\langle \partial_kg-\frac{v_k}{2}g,-\partial_{k}h-\frac{v_k}{2}h\Big\rangle _{L^2(\RR_v^3)},
\\
E_2&=- \ \sum_{\substack{1 \leq k,j \leq 3\\ k \neq j }}\big\langle  f,v_kv_j\mu^{1/2}\big\rangle _{L^2(\RR_v^3)} \Big\langle \partial_jg-\frac{v_j}{2}g,-\partial_{k}h-\frac{v_k}{2}h\Big\rangle _{L^2(\RR_v^3)},
\\
E_3&=\sum_{\substack{1 \leq k,j \leq 3\\ k \neq j }}\langle  f,\mu^{1/2}\rangle _{L^2(\RR_v^3)} \Big\langle  v_k g  ,-\partial_{k}h-\frac{v_k}{2}h\Big\rangle _{L^2(\RR_v^3)},
\end{align*}
\begin{align*}
E_4&=-\sum_{\substack{1 \leq k,j \leq 3\\ k \neq j }}\langle f, v_k\mu^{1/2}\rangle_{L^2(\RR_v^3)}    \Big\langle  g  ,-\partial_{k}h-\frac{v_k}{2}h\Big\rangle _{L^2(\RR_v^3)},
\\
E_5&= -\frac{1}{2}\sum_{\substack{1 \leq k,j \leq 3\\ k \neq j }}\langle f,\mu^{1/2}\rangle  _{L^2(\RR_v^3)}\big\langle  v_j\partial_k g-v_k\partial_jg,v_j\partial_{k}h-v_k\partial_{j}h\big\rangle _{L^2(\RR_v^3)},
\\
E_6&=\sum_{\substack{1 \leq k,j \leq 3\\ k \neq j }}
\langle f,v_j\mu^{1/2}\rangle _{L^2(\RR_v^3)}\Big\langle v_j\partial_kg-v_k\partial_jg,\partial_{k}h+\frac{v_k}{2}h\Big\rangle _{L^2(\RR_v^3)},
\\
E_7&=\sum_{\substack{1 \leq k,j \leq 3\\ k \neq j }}\langle f,v_j\mu^{1/2}\rangle _{L^2(\RR_v^3)} \Big\langle \partial_kg-\frac{v_k}{2}g,v_j\partial_{k}h-v_k\partial_{j}h\Big\rangle_{L^2(\RR_v^3)}.
\end{align*}
This is exactly \eqref{equ3.1}.
\end{proof}

\section{Trilinear estimates with exponential weights}\label{S-4}

\begin{proposition}\label{prop4.1} Let $0 <t \le 1$ and $0 < { c_0 }\ll 1$.
If  $r>\frac 32$, then there exists $C_0 >0$ independent of
$\delta, \delta', c_0$ such that we have for suitable functions $f, g, h$,
\begin{align}\label{tri-estimate}
&\left|
(\Gamma (f, g), F_{\delta, \delta'} \la \delta' v \ra^{-4} F_{\delta, \delta'} h)_{H_x^{r}(L^2_v)}
\right|   \notag \\
&\qquad \leq C_0 \|F_{\delta, 0}\, f\|_{H_x^{r}(L^2_v)}
 \left(||| \la \delta' v \ra^{-2} F_{\delta, \delta'}   g |||_{r, 0}+
\|\la \delta' v \ra^{-2} F_{\delta, \delta'}   g \|_{H_x^{r}(L^2_v)}\right)\\
&\qquad\qquad \qquad\qquad\times \left(||| \la \delta' v \ra^{-2} F_{\delta, \delta'}   h |||_{r, 0}+
\|\la \delta' v \ra^{-2} F_{\delta, \delta'}   h \|_{H_x^{r}(L^2_v)}\right)\,. \notag
\end{align}
\end{proposition}

We consider the nonlinear term
$
(\Gamma(f, { g}), F_{\delta, \delta'} \la \delta' v \ra^{-4} F_{\delta, \delta'}{h})_{H_x^{r}(L^2_v)}$ .
For instance we estimate a term of $D_5$, that is,
\begin{align*}
&\int_{\RR_x^3}(f,\Phi_0)_{L^2(\RR_v^3)}\big(L_{k, j}g,
\la D_x \ra^{2 r} L_{k, j} F_{\delta, \delta'} \la \delta' v \ra^{-4} F_{\delta, \delta'}  h \big)_{L^2(\RR_v^3)} dx \\
&= \int_{\RR^6}  \left( \int_{\RR^3} \hat f(\eta-\tilde \eta, \xi_*)\overline{\hat \Phi_0(\xi_*)} \frac{d\xi_*}{(2\pi)^3}
\hat L_{k, j} \hat g(\tilde \eta, \xi)
d\tilde \eta\right)\\
& \quad \qquad \qquad \times \la \eta \ra^{2 r} \overline{
\hat L_{k, j} F_{\delta, \delta'} \la \delta' D_{\xi} \ra^{-4} F_{\delta, \delta'} \hat h( \eta, \xi) }\frac{d\eta d \xi}{(2\pi)^6}\\
&= \int_{\RR^6}  \left( \int_{\RR^3} \hat f(\eta-\tilde \eta, \xi_*)\overline{\hat \Phi_0(\xi_*)} \frac{d\xi_*}{(2\pi)^3} \hat L_{k, j} \hat g(\tilde \eta, \xi)
d\tilde \eta\right)\\
& \quad \qquad \qquad \times \la \eta \ra^{2 r} \overline{
F_{\delta, \delta'} \hat L_{k, j}\la \delta' D_{\xi} \ra^{-4}  F_{\delta,\delta'} \hat h( \eta, \xi) }\frac{d\eta d \xi}{(2\pi)^6} \\
&\quad +\int_{\RR^6}  \left( \int_{\RR^3} \hat f(\eta-\tilde \eta, \xi_*)\overline{\hat \Phi_0(\xi_*)} \frac{d\xi_*}{(2\pi)^3} \hat L_{k, j} \hat g(\tilde \eta, \xi)
d\tilde \eta\right)\\
& \quad \qquad \qquad \times \la \eta \ra^{2 r} \overline{
 { [\hat L_{k, j}, F_{\delta, \delta'}]}
\la \delta' D_{\xi} \ra^{-4}  F_{\delta,\delta'} \hat h( \eta, \xi) }\frac{d\eta d \xi}{(2\pi)^6},\\
&:= \Gamma_1 + \Gamma_2.
\end{align*}
We consider
\[
F_{\delta, \delta'}(t,\eta,\xi) \la \eta \ra^{r} = \frac{e^{\Psi}}{(1+\delta e^{\Psi})}
\left(\frac{ \la \eta \ra}
{(1+
\delta' \Psi) }\right)^r := F_{\delta, 0}\,\, G_{\delta'}.
\]
By means of  $\la W + V \ra \leq \la W \ra + \la V \ra$, we have
\begin{align*}
&\int_0^t\la \xi-\rho \eta\ra d\rho  \le \int_0^t \la \xi +\xi_* -\rho (\eta-\tilde \eta) -\rho  \tilde \eta\ra d\rho + \int_0^t \la\xi_*\ra d\rho \\
&\qquad \leq
\int_0^t \la \xi_*  -
\rho (\eta-\tilde \eta) \ra d\rho  + \int_0^t \la
\xi-\rho  \tilde \eta \ra d\rho + t \la \xi_*\ra .
\end{align*}
Noting that
$\tilde F_{\delta}(X) = e^X/(1+\delta e^X)$ is an increasing function and that
\[
\tilde F_{\delta}(X+Y) \le 3 \tilde F_{\delta}(X)\tilde F_{\delta}(Y)\,,
\]
we have
\begin{align*}
F_{\delta, 0}(t, \eta, \xi)
\leq 9 F_{\delta, 0}(t, \eta - \tilde \eta, \xi_*) F_{\delta, 0}(t, \tilde \eta, \xi)e^{c_0 t\la\xi_*\ra} \,, 
\end{align*}
Since $\Psi(t, \eta, \xi) \sim c_0 t (1+|\xi|^2 + t^2|\eta|^2)^{1/2}$ and $(1+Y)/(a+bY)$ for any constants $a \geq b >0$ is increasing in $Y$,
there exists a constant $C >0$ independent of $\delta' >0$ and $(t,\xi), \eta, \tilde \eta$ such that
\begin{align*}
G_{\delta'}(t,\eta, \xi) \leq C \frac{\la \eta -\tilde \eta\ra^{r} + \la \tilde \eta \ra^{r}}
{(1+ \delta' \Psi(t, \tilde \eta, \xi))^r}.
\end{align*}
Consequently, for another constant $C_7>0$ independent of $\delta' >0$ and $(t,\xi), \eta, \tilde \eta$
we have
\begin{align}\label{weight-triangle}
&F_{\delta, \delta'}(t,\eta,\xi) \la \eta \ra^{r}\\
&\qquad \leq  C_7 \left(\la \eta -\tilde \eta\ra^{r} + \la \tilde \eta \ra^{r}
\right)
F_{\delta, 0}(t, \eta- \tilde \eta, \xi_*) F_{\delta, \delta'}(t, \tilde \eta, \xi)e^{c_0 t\la\xi_*\ra}\,.\notag
\end{align}

\begin{align*}
\Gamma_1 &= \int_{\RR^6}  \left( \int_{\RR^3} \hat f(\eta-\tilde \eta, \xi_*)\overline{\hat \Phi_0(\xi_*)} \frac{d\xi_*}{(2\pi)^3} \hat L_{k, j} \la \delta' D_{\xi} \ra^{-2}\hat g(\tilde \eta, \xi)
d\tilde \eta\right)\\
& \quad \qquad \qquad \times \la \eta \ra^{2 r} \overline{
\la \delta' D_{\xi} \ra^{2}F_{\delta, \delta'} \hat L_{k, j}\la \delta' D_{\xi} \ra^{-4}  F_{\delta,\delta'} \hat h( \eta, \xi) }\frac{d\eta d \xi}{(2\pi)^6} \\
&= \int_{\RR^6}  \left( \int_{\RR^3} \hat f(\eta-\tilde \eta, \xi_*)\overline{\hat \Phi_0(\xi_*)} \frac{d\xi_*}{(2\pi)^3} \hat L_{k, j} \la \delta' D_{\xi} \ra^{-2}\hat g(\tilde \eta, \xi)
d\tilde \eta\right)\\
& \quad \qquad \qquad \times \la \eta \ra^{2 r} \overline{
F_{\delta, \delta'} \hat L_{k, j}\la \delta' D_{\xi} \ra^{-2}  F_{\delta,\delta'} \hat h( \eta, \xi) }\frac{d\eta d \xi}{(2\pi)^6} \\
&\quad + \int_{\RR^6}  \left( \int_{\RR^3} \hat f(\eta-\tilde \eta, \xi_*)\overline{\hat \Phi_0(\xi_*)} \frac{d\xi_*}{(2\pi)^3}
\hat L_{k, j} \la \delta' D_{\xi} \ra^{-2}\hat g(\tilde \eta, \xi)
d\tilde \eta\right)\\
& \quad \qquad \qquad \times \la \eta \ra^{2 r} \overline{
\left[ \la \delta' D_{\xi} \ra^{2}, F_{\delta, \delta'}\right] \la \delta' D_{\xi} \ra^{-2}\hat L_{k, j}\la \delta' D_{\xi} \ra^{-2}  F_{\delta,\delta'} \hat h( \eta, \xi) }\frac{d\eta d \xi}{(2\pi)^6} \\
&:= \Gamma_{1,1} + \Gamma_{1,2},
\end{align*}
where we have used again $\left[\hat{L}_{k,j}, \la \delta'D_\xi \ra\right] =0$.
By means of \eqref{weight-triangle} and Young inequality,  we get
\begin{align*}
|\Gamma_{1,1}|
&\lesssim  \|\la \cdot \ra^{r} F_{\delta, 0}(\cdot, \xi_*) \hat f(\cdot, \xi_*)\|_{L^2(\RR^6)} \|
e^{c_0 t \la\xi_*\ra - |\xi_*|^2/4}\|_{L^2(\RR^3_{\xi_*})}\\
&\qquad \times \|F_{\delta,\delta'} (\cdot, \xi) \hat L_{k,j} \la \delta' D_{\xi} \ra^{-2} \hat g(\cdot, \xi)\|_{L^2_\xi L^1}\|L_{k,j}
\la \delta' v \ra^{-2}F_{\delta , \delta'}
h \|_{H_x^{3/2 +\epsilon}L^2_v}\\
&\qquad + \|F_{\delta, 0}(\cdot, \xi_*) \hat f(\cdot, \xi_*)\|_{L^2_{\xi_*} L^1} \|
e^{c_0 t \la \xi_* \ra - |\xi_*|^2/4}\|_{L^2(\RR^3_{\xi_*})}\\
&\qquad \times \|\la \cdot \ra^{r} F_{\delta, \delta'}(\cdot, \xi) \hat L_{k,j} \la \delta' D_{\xi} \ra^{-2} \hat g(\cdot, \xi)\|_{L^2(\RR^6)}\|L_{k,j}\la \delta' v \ra^{-2}F_{\delta,\delta'} h \|_{H_x^{r}(L^2_v)}\\
&\lesssim \|F_{\delta, 0} f\|_{H_x^{r}(L^2_v)}{\|F_{\delta, \delta'}  \la \delta' v \ra^{-2}L_{k,j} g \|_{H_x^{r}(L^2_v)}}\|L_{k,j}\la \delta' v \ra^{-2}F_{\delta , \delta'}h \|_{H_x^{r}(L^2_v)}\,.
\end{align*}
{By the same calculus as in \eqref{later-use},
we note that}
\begin{align*}
&F_{\delta,\delta'}(t, \eta, D_v) = \la \delta' v\ra^{-2}
 F_{\delta,\delta'}(t, \eta, D_v) \la \delta' v\ra^2\\
&+2  \delta'^2 \sum_{j=1}^3 v_j \la \delta' v\ra^{-2}  (D_{\xi_j} F_{\delta,\delta'})(t, \eta, D_v)
- \delta'^2  \la \delta' v\ra^{-2} \sum_{j=1}^3(D^2 _{\xi_j} F_{\delta,\delta'})(t, \eta, D_v).
\end{align*}
By means of \eqref{4.3} and \eqref{derive-de-psi}, there exists a constant $C_8 >0$ independent of $t, \eta, \delta,\delta' >0$ such that
for $H(v) \in L^2_2 (\RR^3)$ we have
\begin{align*}%
\|F_{\delta, \delta'}(t, \eta, D_v) H\|_{L^2_v}
\leq &  \| \la \delta' v\ra^{-2} F_{\delta, \delta'}(t, \eta, D_v) \la \delta' v\ra^2 H\|_{L^2_v}\\
&
+ C_8 \delta' c_0  \|F_{\delta, \delta'}(t, \eta, D_v) H\|_{L^2_v},
\end{align*}
and hence
\begin{align*}
\|F_{\delta, \delta'}(t, \eta, D_v) H\|_{L^2_v} \le 2
\|\la \delta' v \ra^{-2} F_{\delta, \delta'}  \la \delta' v \ra^{2}H \|_{L^2_v}
\end{align*}
if $c_0 C_8  < 1/2$. Consequently, we obtain
\begin{align*}
&{\|F_{\delta, \delta'}  \la \delta' v \ra^{-2}L_{k,j} g\|_{H_x^{r}(L^2_v)} }\leq  2
\|\la \delta' v \ra^{-2} F_{\delta, \delta'}  L_{k,j} g\|_{H_x^{r}(L^2_v)} \\
& \qquad \leq   2
\| L_{k,j} \la \delta' v \ra^{-2} F_{\delta, \delta'}   g \|_{H_x^{r}(L^2_v)}  + 2
\|\la \delta' v \ra^{-2}{ \left[ L_{k,j}, F_{\delta, \delta'}  \right] }g\|_{H_x^{r}(L^2_v)} \\
& \qquad \lesssim \| L_{k,j} \la \delta' v \ra^{-2} F_{\delta, \delta'}   g \|_{H_x^{r}(L^2_v)}
 + \| D_{v_k} \la \delta' v \ra^{-2} F_{\delta, \delta'}   g\|_{H_x^{r}(L^2_v)}\\
&\qquad \qquad \qquad + \| D_{v_j} \la \delta' v \ra^{-2} F_{\delta, \delta'}   g \|_{H_x^{r}(L^2_v)}+
\|\la \delta' v \ra^{-2} F_{\delta, \delta'}   g\|_{H_x^{r}(L^2_v)}\,.
\end{align*}
{Here we have used \eqref{ljk-derive} at the third inequality.}
As a consequence, we obtain
\begin{align*}
|\Gamma_{1,1}|
& \lesssim \|F_{\delta, 0} f\|_{H_x^r(L^2_v)}
\left(||| \la \delta' v \ra^{-2} F_{\delta, \delta'}   g|||_{r, 0}+
\|\la \delta' v \ra^{-2} F_{\delta, \delta'}   g\|_{H_x^{r}(L^2_v)}\right)\\
&\qquad\qquad \qquad\qquad\times \left(||| \la \delta' v \ra^{-2} F_{\delta, \delta'}   h |||_{r, 0}+
\|\la \delta' v \ra^{-2} F_{\delta, \delta'}   h \|_{H_x^{r}(L^2_v)}\right).
\end{align*}
Since it follows from the almost same calculation as in \eqref{later-use}, \eqref{later-use-2} and \eqref{later-use-3} that
\begin{align*}
&\left[ \la \delta' D_{\xi} \ra^{2}, F_{\delta, \delta'}\right]\la \delta' D_\xi \ra^{-2}\\
&= \sum_{j=1}^3 \left( (D_{\xi_j}F_{\delta, \delta'})\big(2{\delta' }^2 D_{\xi_j} \la \delta' D_\xi \ra^{-2}\big)
+  (D_{\xi_j}^2 F_{\delta, \delta'}) {\delta' }^2 \la \delta' D_\xi \ra^{-2} \right)\\
&= F_{\delta, \delta'} (t, \eta, \xi) \sum_{j=1}^3\Big (B_{j, \delta, \delta'}(t,\eta, \xi) \big(2{\delta' }^2 D_{\xi_j} \la \delta' D_\xi \ra^{-2}\big)
+ \tilde B_{j, \delta, \delta'} (t, \eta, \xi){\delta'}^2\la \delta' D_\xi \ra^{-2}\Big)\,,
\end{align*}
{and since the last factor is a bounded operator, }
the estimation for $\Gamma_{1,2}$ is similar to the one for $\Gamma_{1,1}$,

{As for } the estimation of $\Gamma_2$, we recall \eqref{ljk-derive}. Then
\begin{align*}
&[\hat L_{k, j}, F_{\delta, \delta'}]\la \delta' D_\xi \ra^{-2} = i F_{\delta, \delta'} \Big( B_{k, \delta, \delta'}\xi_j
-  B_{j, \delta, \delta'} \xi_k \Big)
\la \delta' D_\xi \ra^{-2} \\
&=
i F_{\delta, \delta'} \left( B_{k, \delta, \delta'}
\la \delta' D_\xi \ra^{-2} \Big( \xi_j - 2i \frac{{\delta'}^2 D_{\xi_j}}{\la \delta' D_\xi \ra^{2}}\Big)
-  B_{j, \delta, \delta'}
\la \delta' D_\xi \ra^{-2} \Big( \xi_k - 2i \frac{{\delta'}^2 D_{\xi_k}}{\la \delta' D_\xi \ra^{2}}\Big)\right).
\end{align*}
Writing
\[
F_{\delta, \delta'} B_{k, \delta, \delta'}
\la \delta' D_\xi \ra^{-2} = F_{\delta, \delta'}\la \delta' D_\xi \ra^{-2} \Big (B_{k, \delta, \delta'} +
[\la \delta' D_\xi \ra^2, B_{k, \delta, \delta'}] \la \delta' D_\xi \ra^{-2} \Big),
\]
we see that the estimation for  $\Gamma_2$ is quite similar to {the one for  $\Gamma_1$, because there exist bounded operators $R_k, R_j,
R_{j,k}$ such that
\[
[\hat L_{k, j}, F_{\delta, \delta'}]\la \delta' D_\xi \ra^{-2} = F_{\delta, \delta'}\la \delta' D_\xi \ra^{-2} \Big(R_k \xi_j +
R_j \xi_k + R_{j,k}\Big)\,.
\] }
Thus we have a desired estimate for $D_5$. For the other terms $D_j$, we remark that
$\left[\hat{a}_{\pm j}, \la \delta'D_\xi \ra\right] \not=0$, but it is bounded, whence we also can estimate them by the same procedure.
We omit the detail.


\section{End of proof of main theorem}
In the first subsection we show the local existence of solution in $[0,1]$ and its analytic smoothing effect.
We remark the stability and uniqueness of this local solution in the next subsection. In the last subsection we complete
the proof of the main theorem by using the global existence theorem given by Guo\cite{Guo}.

{\color{black}

\subsection{Existence of analytic time-local solution}\label{s5-1}

\begin{lemma}%
[local existence for a linear equation]\label{linear-equation}
Let $ r >3/2$ and $0 < c_0 \ll 1$. Assume that $0 < \delta \le 1$. Then
there exist $\epsilon_0>0$ and $C_9 >1$ independent of $\delta$ such that for any $0 < T \le 1$, $g_0 \in H_x^r (L^2_v)$, $f \in
 L^\infty([0,T]; H_x^r(L^2_v))$
satisfying
\begin{align}\label{inductive-1}
 \|F_{\delta, 0} f\|_{ L^\infty([0,T]; H_x^r(L^2_v))} \leq \epsilon_0,
\end{align}
the Cauchy problem
\begin{align}\label{l-e-c}\begin{cases}
\partial_tg+v\cdot \nabla_{x}g+\mathcal{L}g=\Gamma(f,g) 
,\\
g|_{t=0}=g_0,
\end{cases}\end{align}
admits a weak solution $g \in  L^\infty([0,T]; H_x^r(L^2_v))$ satisfying
\begin{align}\label{energy-important}
\|F_{\delta, 0}g\|^2_{ L^\infty([0,T]; H_x^r(L^2_v))}
+
\int_0^T ||| F_{\delta, 0} g(s) |||^2_{r, 0}ds
\le  C_9\|g_0\|^2_{H_x^r(L^2_v)}
. 
\end{align}

\end{lemma}

\begin{proof}
Consider
$$\mathcal{Q}=-\partial_t+(v\cdot \nabla_{x}+\mathcal{L}-\Gamma(f,\cdot))^*,$$
where the adjoint operator $(\cdot)^*$ is taken with respect to the scalar product in $H_x^r(L^2_v)$.
Then, by using \eqref{2.18} and \eqref{tri-estimate} with $\delta, \delta',  c_0 \rightarrow 0$ we see that
for all $h \in C^{\infty}([0,T], \mathcal{S}(\RR_{x,v}^6))$, with $h(T)=0$ and~$0 \leq t \leq T$,
\begin{align*}
& \textrm{Re}\big(h(t),\mathcal{Q}h(t)\big)_{H_x^r(L^2_v)}=   -\frac{1}{2}\frac{d}{dt}(\|h\|^2_{H_x^r(L^2_v)})\\ \notag
&  \quad +\textrm{Re}(v\cdot\nabla_{x}h,h)_{H_x^r(L^2_v)}+\textrm{Re}(\mathcal{L}h,h)_{H_x^r(L^2_v)}-\textrm{Re}(\Gamma(f,h),h)_{H_x^r(L^2_v)} \\ \notag
\geq& \ -\frac{1}{2}\frac{d}{dt}\big(\|h(t)\|^2_{H_x^r(L^2_v)} \big)
+\frac{1}{4}|||h(t)|||^2_{r,0}- C_8 \|h(t)\|_{H_x^r(L^2_v)}^2\\
& \qquad -C_0\|f(t)\|_{H_x^r(L^2_v)}
|||h(t)|||^2_{r,0}\,,
\end{align*}
because $\mathcal{L}$ is a selfadjoint operator and $\textrm{Re}(v\cdot \nabla_{x}h,h)_{H_x^r(L^2_v)}=0$.
Since \eqref{inductive-1} implies  $\| f\|_{ L^\infty([0,T]; H_x^r(L^2_v))} \leq 2 \epsilon_0$,
we have
\begin{align*}
-\frac{d}{dt}\big(e^{2C_8t}\|h(t)\|_{H_x^r(L^2_v)}^2\big)+&\frac{1}{4}e^{2C_8t}|||h(t)|||^2_{r,0} \\
&\leq 2e^{2C_8t}\|h(t)\|_{H_x^r(L^2_v)}\|\mathcal{Q}h(t)\|_{H_x^r(L^2_v)},
\end{align*}
if $16 \varepsilon_0C_0 < 1$. Since $h(T)=0$, for all $t \in [0,T]$ we have
\begin{align*}
& \ \|h(t)\|_{H_x^r(L^2_v)}^2+\frac{1}{4}
\int_t^T |||h(\tau)|||^2_{r,0}d \tau \\
&\quad \leq  \  2\int_t^Te^{2C_8(\tau-t)}\|h(\tau)\|_{H_x^r(L^2_v)}\|\mathcal{Q}h(\tau)\|_{H_x^r(L^2_v)}d\tau \\
&\quad \leq 2e^{2C_8T}\|h\|_{L^{\infty}([0,T] ;H_x^r(L^2_v))}\|\mathcal{Q}h\|_{L^{1}([0,T],H_x^r(L^2_v))}, \enskip \text{so that}
\end{align*}
\begin{equation}\label{tr7} \|h\|_{L^{\infty}([0,T] ;H_x^r(L^2_v))} \leq 2 e^{2C_8T}\|\mathcal{Q}h\|_{L^{1}([0,T],H_x^r(L^2_v))}.
\end{equation}
We
consider the vector subspace
\begin{align*}
\mathbb{W}&=\{w=\mathcal{Q}h : h \in C^{\infty}([0,T],\mathcal{S}(\RR_{x,v}^6)), \ h(T)=0\} \\
&\subset L^{1}([0,T],
H_x^r(L^2_v)).
\end{align*}
This inclusion holds because  it follows from Proposition \ref{prop3.1} that for $g \in H_x^r(L^2_v)$
\begin{align*}
|(\Gamma(f,\cdot)^*h,g)_{H_x^r(L^2_v)}|=|(h,\Gamma(f,g))_{H_x^r(L^2_v)}|
\lesssim \|f \|_{H_x^r(L^2_v)}
\|g\|_{H_x^r(L^2_v)} \|\la v \ra^2 h\|_{H^r_x(H^{2}_v)}\,,
\end{align*}
and hence, for all $t \in [0,T]$,
$$\|\Gamma(f,\cdot)^*h\|_{H_x^r(L^2_v)}  \lesssim
\|f \|_{H_x^r(L^2_v)}\|\la v \ra^2 h\|_{H^r_x(H^{2}_v)}.$$
Since $g_0 \in H_x^r(L^2_v)$,  we define
the linear functional
\begin{align*}
\mathcal{G} \ : \qquad &\mathbb{W} \enskip \longrightarrow \CC\\ \notag
w=&\mathcal{Q}h \mapsto (g_0,h(0))_{H_x^r(L^2_v)} 
\end{align*}
where $h \in C^{\infty}([0,T],\mathcal{S}(\RR_{x,v}^6))$, with $h(T)=0$.
According to (\ref{tr7}), the operator $\mathcal{Q}$ is injective. The linear functional $\mathcal{G}$ is therefore well-defined. It follows from
(\ref{tr7}) that $\mathcal{G}$ is a continuous linear form on $(\mathbb{W},\|\cdot\|_{L^{1}([0,T]; H_x^r(L^2_v))})$,
\begin{align*}
|\mathcal{G}(w)| &\leq \|g_0\|_{H_x^r(L^2_v)}\|h(0)\|_{H_x^r(L^2_v)}
\\
&\le 2 e^{2C_8T}\|g_0\|_{H_x^r(L^2_v)}
\|\mathcal{Q}h\|_{L^1([0,T];H_x^r(L^2_v))}
\\
& = 2 e^{2C_8T}\|g_0\|_{H_x^r(L^2_v)}
\|w\|_{L^1([0,T];H_x^r(L^2_v))}\,.
\end{align*}
By using the Hahn-Banach theorem, $\mathcal{G}$ may be extended as a continuous linear form on
$$L^{1}([0,T]; H_x^r(L^2_v)),$$
with a norm smaller than
$2 e^{2C_8T}
 \|g_0\|_{H_x^r(L^2_v)}$. 
Hence  there exists $g \in L^{\infty}([0,T]; H_x^r(L^2_v))$ satisfying
$$\|g\|_{L^{\infty}([0,T], H_x^r(L^2_v))} \leq
 2 e^{2C_8T}\|g_0\|_{H_x^r(L^2_v)} 
,$$
such that
$$\forall w \in L^{1}([0,T]; H_x^r(L^2_v)), \quad \mathcal{G}(w)=\int_0^T(g(t),w(t))_{H_x^r(L^2_v)}dt.$$
This implies
that for all $h \in
C_0^{\infty}((-\infty,T),\mathcal S(\RR_{x,v}^6))$,
\begin{align*}
\mathcal{G}(\mathcal{Q}h)&=\int_0^T(g(t),\mathcal{Q}h(t))_{H_x^r(L^2_v)}dt\\
&=(g_0,h(0))_{H_x^r(L^2_v)}\,.
\end{align*}
This shows that $g \in L^{\infty}([0,T]; H_x^r(L^2_v))$ is a weak solution of the Cauchy problem
\eqref{l-e-c}.

It remains to show \eqref{energy-important}. Noting  that
$g \in L^{\infty}([0,T]; H_x^r(L^2_v))$  implies,
for any $\delta >0, \delta' >0$,
 \[
\la v \ra \big( t^r  \la D_v \ra^r + t^{2r} \la D_x \ra^r \big)\la \delta' v\ra^{-2}
F_{\delta, \delta'}g \in L^\infty([0, T]; H_x^r(L^2_v)),
\]
we multiply the first equation of \eqref{l-e-c} by
$F_{\delta, \delta'}\la \delta' v\ra^{-4}
F_{\delta, \delta'}g$ and take  its $H_x^r(L^2_v)$ inner product.
Then it follows from \eqref{2.18} and \eqref{tri-estimate}
that for $0 < t \le T \le 1$
\begin{align}\label{fundamental}
\frac{1}{2}\frac{d}{dt}
&\|\la \delta' v \ra^{-2} F_{\delta, \delta'} g\|^2_{H^r_x(L^2_v)}+\frac 14||| \la \delta' v \ra^{-2} F_{\delta, \delta'} g |||^2_{r, 0}
 -C_8\|\la \delta' v \ra^{-2} F_{\delta, \delta'} g \|^2_{H_x^{r}( L^2_v)}\\
&\le 2 C_0 \|F_{\delta, 0}\, f\|_{H_x^{r}(L^2_v)}
 \left(||| \la \delta' v \ra^{-2} F_{\delta, \delta'}   g |||^2_{r, 0}+
\|\la \delta' v \ra^{-2} F_{\delta, \delta'}   g \|^2_{H_x^{r}(L^2_v)}\right)\,.\notag
\end{align}
Since $\epsilon_0$ is chosen small enough that $16\epsilon_0 C_0 <1$,  we get
\begin{align*}
&\|\la \delta' v \ra^{-2} F_{\delta, \delta'} g\|^2_{ L^\infty([0,T]; H_x^r(L^2_v))}
+\frac{1}{4}
\int_0^T ||| \la \delta' v \ra^{-2} F_{\delta, \delta'}  g(s) |||^2_{r, 0}ds\\
& \qquad \qquad \qquad \le  2  e^{3C_8T} \|\la \delta' v \ra^{-2} g_0\|^2_{H_x^r(L^2_v)}\,,%
\end{align*}
which yields \eqref{energy-important} with $C_9 = 8 e^{3C_8}$, by letting $\delta' \rightarrow 0$.
\end{proof}

\begin{theorem} [analytic time-local solution]\label{local-thm}
Let $r>3/2$. There exists an $\epsilon_1>0$ 
such that for all $g_0 \in H^{r}_x(L^2_v)$ satisfying
$$
\|g_0\|_{H^{r}_x(L^2_v)} \leq \epsilon_1,
$$
the Cauchy problem  \eqref{landau-2} admits a solution such that
$$
g(t)\in \cA (\RR^6_{x, v}),\quad  0< \forall t \le  1\,.
$$
Furthermore, there exists a $0 <c_1 <1$ such that,
\begin{align*}
e^{c_1\{t^2 (-\Delta_x)^{1/2}+ t (-\Delta_v)^{1/2}\} }g(t)\in L^{\infty}\bigl([0,1],  H^{r}_x(L^2_v)\bigr),
\end{align*}
more precisely for any $0<t \le 1$
\begin{align}\label{est-bien}
\|e^{c_1\{t^2 (-\Delta_x)^{1/2}+ t (-\Delta_v)^{1/2}\} }g(t)||_{H^{r}_x(L^2_v)} \le \sqrt{C_9} \|g_0\|_{H^{r}_x(L^2_v)} .
\end{align}
\end{theorem}

\begin{proof}
Consider the sequence of approximate solution defined by
\begin{equation}\label{landau-n}
\begin{cases}
\partial_t g^{n+1}+v\cdot\nabla_{x} g^{n+1}+\cL g^{n+1}=\Gamma (g^n,\,  g^{n+1}).\\
g^{n+1}|_{t=0}=g_0,
\end{cases}
\end{equation}
with
$$
g^0=e^{-\Psi(t, D_x, D_v)} g_0.
$$
We apply Lemma \ref{linear-equation} with $f = g^n$ and $g = g^{n+1}$ by assuming $\sqrt{C_9} \epsilon_1 \le \epsilon_0$.
Then it follows from \eqref{energy-important} with $T =1$ that for $n \ge 1$
\begin{equation}\label{upper-one}
\|F_{\delta,0} g^n\|^2_{L^\infty([0,1]; H_x^r(L^2_v))}  +\int_0^1|||F_{\delta,0} g^n(s)|||^2_{r,0}ds\le
C_9 \|g_0\|^2_{H_x^r(L^2_v)} \le  \epsilon_0^2
\end{equation}
holds inductively
because
$
\|F_{\delta,0} g^0\|_{H_x^r(L^2_v)} \le \|g_0\|_{H_x^r(L^2_v)} \le \epsilon_1 \le \epsilon_0\,.
$
Setting $w^n = g^{n+1} - g^n$, from \eqref{landau-n} we have
\[
\partial_t w^{n}+v\cdot \nabla_{x}w^{n}+\mathcal{L}w^{n} =\Gamma(g^n, w^n)+
\Gamma(w^{n-1},g^{n})\,,
\]
with $w^{n}|_{t=0}=0$. Similar to the computation for \eqref{fundamental}, we obtain
\begin{align*}
\frac{1}{2}\frac{d}{dt}
&\|\la \delta' v \ra^{-2} F_{\delta, \delta'} w^n\|^2_{H^r_x(L^2_v)}+\frac 14||| \la \delta' v \ra^{-2} F_{\delta, \delta'} w^n |||^2_{r, 0}
 -C_8\|\la \delta' v \ra^{-2} F_{\delta, \delta'} w^n \|^2_{H_x^{r}( L^2_v)}\\
&\le 2 C_0 \|F_{\delta, 0}\, g^n\|_{H_x^{r}(L^2_v)}
 \left(||| \la \delta' v \ra^{-2} F_{\delta, \delta'}   w^n |||^2_{r, 0}+
\|\la \delta' v \ra^{-2} F_{\delta, \delta'}   w^n \|^2_{H_x^{r}(L^2_v)}\right)\\
& + 2C_0 \|F_{\delta, 0}\, w^{n-1}\|_{H_x^{r}(L^2_v)}
 \left(||| \la \delta' v \ra^{-2} F_{\delta, \delta'}   g^n |||^2_{r, 0}+
\|\la \delta' v \ra^{-2} F_{\delta, \delta'}   g^n \|^2_{H_x^{r}(L^2_v)}\right)^{1/2}\\
&\qquad \qquad \times \left(||| \la \delta' v \ra^{-2} F_{\delta, \delta'}   w^n |||^2_{r, 0}+
\|\la \delta' v \ra^{-2} F_{\delta, \delta'}   w^n \|^2_{H_x^{r}(L^2_v)}\right)^{1/2}\\
\le &\Big(2 C_0 \|F_{\delta, 0}\, g^n\|_{H_x^{r}(L^2_v)} + \frac{1}{16}\Big)
 \left(||| \la \delta' v \ra^{-2} F_{\delta, \delta'}   w^n |||^2_{r, 0}+
\|\la \delta' v \ra^{-2} F_{\delta, \delta'}   w^n \|^2_{H_x^{r}(L^2_v)}\right)\\
&\quad + 32C_0^2 \|F_{\delta, 0}\, w^{n-1}\|^2_{L^\infty([0,1]; H_x^{r}(L^2_v))}\\
 & \qquad \times \left(||| \la \delta' v \ra^{-2} F_{\delta, \delta'}   g^n |||^2_{r, 0}+
\|\la \delta' v \ra^{-2} F_{\delta, \delta'}   g^n \|^2_{H_x^{r}(L^2_v)}\right)\,,
\end{align*}
which implies
\begin{align*}
&\|\la \delta' v \ra^{-2} F_{\delta, \delta'} w^n\|^2_{L^\infty([0,1]; H^r_x(L^2_v))} +
\frac{1}{8}\int_0^1||| \la \delta' v \ra^{-2} F_{\delta, \delta'} w^n(\tau) |||^2_{r, 0}d\tau\\
&\le 64 C_0^2 e^{3C_8}\Big(\int_0^1|||  F_{\delta, 0}   g^n (\tau)|||^2_{r, 0}d\tau
+  \|F_{\delta, 0}   g^n \|^2_{L^\infty([0,1];H_x^{r}(L^2_v))}\Big)\\
&\qquad \qquad \qquad \qquad \qquad \times \|F_{\delta, 0}\, w^{n-1}\|^2_{L^\infty([0,1]; H_x^{r}(L^2_v))}\\
&\le 64 C_0^2 e^{3C_8}\times 2\epsilon_0^2 \|F_{\delta, 0}\, w^{n-1}\|^2_{L^\infty([0,1]; H_x^{r}(L^2_v))}\,,
\end{align*}
because of \eqref{upper-one}  and
$16 \varepsilon_0C_0 < 1$. Letting $\delta' \rightarrow 0$, we see that there exists a $0 < \lambda <1$ such that
\[
\|F_{\delta, 0} w^n\|^2_{L^\infty([0,T]; H^r_x(L^2_v))}
\le \lambda \|F_{\delta, 0}\, w^{n-1}\|^2_{L^\infty([0,T]; H_x^{r}(L^2_v))}
\]
if $\epsilon_0>0$ is small enough so that $128 C_0^2 e^{3C_8}\epsilon_0^2 \le \lambda$.
By taking $\delta \rightarrow 0$ we see that there exists a local solution $g \in L^\infty([0,T]; H_x^r(L^2_v))$
of the Cauchy problem \eqref{landau-2} such that
\[
\|e^{\Psi} (g^n -g ) \|_{L^\infty([0,1]; H_x^{r}(L^2_v))} \rightarrow 0 \enskip \mbox{as} \enskip n \rightarrow \infty,
\]
and
\[
\|e^\Psi g\|^2_{L^\infty([0,1]; H_x^r(L^2_v))}  +\int_0^1|||e^\Psi g(s)|||^2_{r,0}ds\le C_9 \|g_0\|^2_{H_x^r(L^2_v)} \le  \epsilon_0^2\,.
\]
By means of Lemma \ref{lemm4.1}, we get the desired estimate \eqref{est-bien}.
\end{proof}

\subsection{The stability and uniqueness of time-local solution}

\begin{proposition}[Stability and uniqueness]\label{uniquenss}
Let $r >3/2, 0<c_0 \ll 1$ and let $\epsilon_0 >0$ satisfy $16 C_0 \epsilon_0 <1$ for the constant $C_0 $ in
Proposition \ref{prop4.1}.  Let $0<T \le1$ and
let $g_j(t) \in L^\infty([0,T]; H_x^r(L^2_v))$, $j=1,2$ be two solutions of the Cauchy problem
\eqref{landau-2} with initial data $g_{1,0} ,g_{2,0} \in H_x^r(L^2_v)$, respectively.
Assume that
$$
\|e^\Psi g_1 \|_{L^\infty([0,T]; H_x^r(L^2_v))} \le \epsilon_0\,,
$$
and there exists $M >0$,  such that
$$
\|e^\Psi g_2\|^2_{L^\infty([0,T]; H_x^r(L^2_v))}  +\int_0^T|||e^\Psi g_2(s)|||^2_{r,0}ds\le M.
$$
Then there exists a $C_M >1$ independent of $c_0$ such that
\begin{align}\label{stability}
&\|e^\Psi(g_1-g_2)\|^2_{L^{\infty}([0,T]; H_x^r(L^2_v))} \\
&\notag \qquad \qquad +
\int_0^T|||e^\Psi(g_1- g_2)(s)|||^2_{r,0}ds
\le C_M \|g_{0,1} - g_{0,2} \|^2_{H_x^r(L^2_v)}\,.
\end{align}
\end{proposition}
\begin{remark}\label{rema-unique}
It is easy to see that this proposition holds in the case where $c_0=0$, that is, $e^\Psi$ is replaced by $1$.
Therefore the uniqueness of solutions belonging to
${L^\infty([0,T]; H_x^r(L^2_v))}$ holds under above conditions with $e^\Psi =1$.
\end{remark}

\begin{proof}
Putting $w = g_1-g_2$, we have
\[
\partial_t w+v\cdot \nabla_{x}w+\mathcal{L}w =\Gamma(g_1, w)+
\Gamma(w,g_2)\,,
\]
with an initial datum $w|_{t=0} = g_{0,1} -g_{0,2} \in H_x^r(L^2_v)$.
Similar to the proof of Theorem \ref{local-thm}, we obtain
\begin{align*}
\frac{1}{2}\frac{d}{dt}
&\|\la \delta' v \ra^{-2} F_{\delta, \delta'} w\|^2_{H^r_x(L^2_v)}+\frac 14||| \la \delta' v \ra^{-2} F_{\delta, \delta'} w|||^2_{r, 0}
 -C_8\|\la \delta' v \ra^{-2} F_{\delta, \delta'} w \|^2_{H_x^{r}( L^2_v)}\\
&\le \Big(2 C_0 \|F_{\delta, 0}\, g_1\|_{H_x^{r}(L^2_v)} + \frac{1}{16}\Big)
 \left(||| \la \delta' v \ra^{-2} F_{\delta, \delta'}   w |||^2_{r, 0}+
\|\la \delta' v \ra^{-2} F_{\delta, \delta'}   w \|^2_{H_x^{r}(L^2_v)}\right)\\
&\quad + 32C_0^2 \|F_{\delta, 0}\, w\|^2_{H_x^{r}(L^2_v)}\\
 & \qquad \times \left(||| \la \delta' v \ra^{-2} F_{\delta, \delta'}   g_2 |||^2_{r, 0}+
\|\la \delta' v \ra^{-2} F_{\delta, \delta'}   g_2\|^2_{H_x^{r}(L^2_v)}\right)\,,
\end{align*}
which implies
\begin{align*}
\frac{d}{dt}
&\|\la \delta' v \ra^{-2} F_{\delta, \delta'} w\|^2_{H^r_x(L^2_v)}+\frac 18||| \la \delta' v \ra^{-2} F_{\delta, \delta'} w|||^2_{r, 0}
 -3C_8\|\la \delta' v \ra^{-2} F_{\delta, \delta'} w \|^2_{H_x^{r}( L^2_v)}\\
&\quad \le 64C_0^2 \|F_{\delta, 0}\, w\|^2_{H_x^{r}(L^2_v)}
\left(||| \la \delta' v \ra^{-2} F_{\delta, \delta'}   g_2 |||^2_{r, 0}+M \right)\,.
\end{align*}
By integrating from $0$ to $t \in(0, T]$, we get
\begin{align*}
&\|\la \delta' v \ra^{-2} F_{\delta, \delta'} w(t)\|^2_{H^r_x(L^2_v)}+\frac 18
\int_0^t ||| \la \delta' v \ra^{-2} F_{\delta, \delta'} w(\tau) |||^2_{r, 0}d\tau \\
& \quad\le \|w(0)\|^2_{H^r_x(L^2_v)} + 64e^{3C_8T_0}C_0^2\int_0^t \|F_{\delta, 0}\, w(\tau)\|^2_{H_x^{r}(L^2_v)}
\Big( M+|||e^{\Psi}g_2(\tau)|||^2_{r,0}\Big)d\tau
\end{align*}
Let $\delta' \rightarrow 0$ and
denote the right hand side by $Y_{\delta}(t)$.  Then
\[
Y'_\delta(t) \le 64e^{3C_8T_0}C_0^2 \Big( M+\|e^{\Psi}g_2(t)|||^2_{r,0}\Big) Y_\delta (t) , \enskip a.e. \enskip t \in [0,T],
\]
so that we obtain
\[
Y_{\delta}(t) \le \|w(0)\|^2_{H^r_x(L^2_v)}  e^{64e^{3C_8T_0}C_0^2(MT+ \int_0^T|||e^\Psi g_2(s)|||^2_{r,0}ds)}\,.
\]
Finally, letting $\delta \rightarrow 0$ we obtain the desired estimate \eqref{stability}.
\end{proof}

}
\subsection{Time global solution and its analytic smoothing}
To show the existence of a global solution and its analyticity smoothing, we refer  the following theorem
that was first proved by Guo\cite{Guo} in the torus $\mathbb{T}_x^3$ case and was extended to the whole space
$\mathbb{R}_x^3$ by Yang-Yu\cite{Y-Yu};

\begin{theorem}[\cite{Guo, Y-Yu}]\label{global-thm}
There exist some positive constants $\epsilon_2>0$, $C_{10}>1$ such that if $g_0  \in H^8(\RR^6_{x,v})$ satisfies
	$\|g_0\|_{H^8_{x,v}}\le \epsilon_2$ then the Cauchy problem \eqref{landau-2} admits a unique global solution
$g(t) \in L^\infty([0, \infty); H^8_{x.v}(\RR^6))$ fulfilling
\[
\sup_{0\le t < \infty}\|g(t) \|_{H^8_{x,v}} \le C_{10} \|g_0\|_{H^8_{x,v}}.
\]
\end{theorem}

Assume that  $0 < \epsilon_3 \le (c_1)^8/(8! C_9) \min \{ \epsilon_1/C_{10}, \epsilon_2\}$ and
$\|g_0\|_{H^r_x(L^2_v)} \le \epsilon_3$.  Let $1 \le \tau \le 2$ and apply
Theorem \ref{local-thm} with the initial time $t_0 = \tau -1$,  in view of $ \sqrt{C_9} \epsilon_3 < \epsilon_1$.
Then for any $\tau \in [1,2]$ we have
\begin{align}\label{1-2}
&\|e^{c_1((-\Delta_x)^{1/2} + (-\Delta_v)^{1/2})} g(\tau)\|_{H^r_x(L^2_v)}\notag\\
&\qquad=\|e^{c_1(((\tau-t_0)^2(-\Delta_x)^{1/2} + (\tau-t_0)(-\Delta_v)^{1/2})} g(\tau)\|_{H^r_x(L^2_v)}\\
&\quad \le \sqrt{ C_9}
\|g(t_0)\|_{H^r_x(L^2_v)} \le \sqrt{ C_9} \epsilon_3,\notag
\end{align}
which implies the existence of local solution $g(t) \in L^\infty([0,2]; H_x^r(L^2_v))$ satisfying
\[
\sup_{[1,2]}\|g(t)\|_{H_{x,v}^8} \le \frac{8!} {(c_1)^8}\sqrt{ C_9} \epsilon_3  \le \epsilon_2.
\]
By using  Theorem \ref{global-thm} with the initial time $t_1=1$, we obtain a global solution
$g(t) \in  L^{\infty}([1,\infty[; H^8_{x,v}(\mathbb{R}^6))$ satisfying
\[
\sup_{[1, \infty[}\|g(t)\|_{H^r_x(L^2_v)} \le
\sup_{[1, \infty[}\|g(t)\|_{H^8_{x,v}} \le \frac{8!} {(c_1)^8}\sqrt{C_9} C_{10}\epsilon_3 \le \epsilon_1.
\]
For $\tau >2$, apply again Theorem \ref{local-thm} with the initial time $t_0= \tau-1$. Then we obtain
\[
\sup_{\tau \ge  2} \|e^{c_1((-\Delta_x)^{1/2} + (-\Delta_v)^{1/2})} g(\tau)\|_{H^r_x(L^2_v)} \le \sqrt{ C_9}\sup_{\tau \ge  2}
\|g(\tau-1)\|_{H^r_x(L^2_v)} \le \sqrt{ C_9}\epsilon_1.
\]
This together with \eqref{1-2} and Theorem \ref{local-thm} complete the proof of \eqref{analytic-smooth}.

\subsection{Global stability and uniqueness}
It follows from Proposition \ref{uniquenss} and the stability of global solutions in Theorem \ref{global-thm} that
the stability of analytic global solutions holds. In fact, for any fixed $T >2$, if $t \in [2, T]$ then
\begin{align*}
\|e^{c_1((-\Delta_x)^{1/2} + (-\Delta_v)^{1/2})}& (g_1(t) - g_2(t))\|_{H_x^r(L^2_v)}  \le \sqrt{C_M} \|g_1(t-1) - g_2(t-1)\|_{H_x^r(L^2_v)}\\
&  \le \sqrt{C_M} C_T  \|g_1(1) - g_2(1)\|_{H_{x,v}^8} \\
&\le  \sqrt{C_M} \frac{8! C_T}{(c_1)^8} \|e^{c_1((-\Delta_x)^{1/2} + (-\Delta_v)^{1/2})}(g_1(1) - g_2(1))\|_{H_x^r(L^2_v)}\\
&\le (C_M)  \frac{8! C_T}{(c_1)^8}\|g_{0,1} - g_{0,2}\|_{H_x^r(L^2_v)}.
\end{align*}
The case $1\le t \le 2$ is an easy consequence of  Proposition \ref{uniquenss}. Thus the uniqueness of global solutions holds. Now the proof of Theorem \ref{theorem-1} is complete.

\bigskip
\noindent
{\bf Acknowledgements.}
The research of the first author is supported by
JSPS Kakenhi
Grant No.17K05318.
The research of the second author is supported partially by ``The Fundamental Research Funds for Central Universities''.
Both authors would like to thank
Nicolas Lerner and Karel Pravda-Starov for stimulating discussions on the topic of smoothing effect for kinetic equations.


\begin{thebibliography}{aa}

\bibitem{ADVW}
R. Alexandre, L. Desvillettes, C. Villani and B. Wennberg, \textit{Entropy dissipation and long-range interactions},
Arch. Ration. Mech. Anal., \textbf{152} (2000), 327--355.


\bibitem{AMUXY1}
R. Alexandre, Y. Morimoto, S. Ukai, C.-J. Xu, T. Yang, \textit{Regularizing effect and local existence for the non-cutoff Boltzmann equation}, Arch. Ration. Mech. Anal. 198 (2010), no. 1, 39-123

\bibitem{AMUXY2}
R. Alexandre, Y. Morimoto, S. Ukai, C.-J. Xu, T. Yang, \textit{The Boltzmann equation without angular cutoff in the whole space: qualitative properties of solutions}, Arch. Ration. Mech. Anal. 202 (2011), 599-661

\bibitem{AMUXY3}
R. Alexandre, Y. Morimoto, S. Ukai, C.-J. Xu, T. Yang, \textit{Global existence and full regularity of the Boltzmann equation without angular cutoff}, Comm. Math. Phys. \textbf{304} (2011), no. 2, 513-581


\bibitem{Alexandre-Villani}
R. Alexandre, C. Villani, \textit{On the Landau approximation in plasma physics},
Ann. Inst. H.  Poincar\'e  Anal. Nonlin\'eaire, \textbf{21} (2004), 61-95.

\bibitem{BHRV}
H. Barbaroux, D. Hundermark, T. Ried, S. Vugalter,
\textit{Gevrey smoothing for weak solutions of the fully nonlinear homogeneous Boltzmann and Kac equation
without cutoff for Maxwellian molecules},
Arch. Rational Mech. Anal., \textbf{225} (2017), 601-661.

\bibitem{Boby0}
A.V.~Bobylev,~~\textit{Expansion of the Boltzmann collision integral in a Landau series}, Sov. Phys. Dokl., \textbf{20} (1976), 740-742.

\bibitem{BGP}
A.V. Bobylev, I.M. Gamba, I.F. Potapenko,
\textit{On some properties of the Lnadau kinetic equation},
J. Stat. Phys., \textbf{161} (2015), 1327-1338.


\bibitem{BPS}
A.V. Bobylev, M. Pulvirenti, C. Saffrio, \textit{From practical system to the Landau equation:
a consistency result}, Commun. Math. Phys., \textbf{319}(2013), 683-702.

\bibitem{bgp}
F. Bouchut, F. Golse and M. Pulvirenti, Kinetic equations and asymtotic theory, editions  scientifiques et
m\'edicales Elsevier, Paris (2000)


\bibitem{CSS} S. Cameron, L. Silvestre and S. Snelson, \textit{
Global a priori estimates for the inhomogeneous Landau equation with moderately soft potentials},
Ann. Inst. H. Poincar\'e Anal. Non lin\'eaire, \textbf{35} (2018), 625-642.


\bibitem{CLXX} H.-M. Cao, H.-G. Li, C.-J. Xu and J. Xu, \textit{Well-posedness of Cauchy problem for Landau equation in critical Besov space}  Kinetic and Related Models  \textbf{12} (2019), No. 4 ~829-884.

\bibitem{CA-MIS}
K. Carrapatoso and S. Mischler, Landau equation for very soft and Coulomb potentials near Maxwellians,
\textit{Ann. Partial Differential Equations} {\bf 3} (2017), 1-65.
\bibitem{CA-TR-Wu}
K. Carrapatoso, I. Tristani and K.-C. Wu, \textit{Cauchy problem and exponential stability for the inhomogeneous Landau equation},
Arch. Ration. Mech. Anal.,\textbf {221} (2016), 363-418

\bibitem{D1992} L. Desvillettes, \textit{On asymptotics of the Boltzmann equation when collisions became grazing},
Transp. Theory Stat. Phys., \textbf{21} (1992), 259-276.

\bibitem{D2015}
L. Desvillettes, \textit{Entropy dissipation estimates for the Landau equation in the Coulomb case and applications}
J. Funct. Anal., \textbf{269} (2015), 1359-1403.

\bibitem{DV1}
L. Desvillettes, C. Villani, \textit{On the spatially homogeneous Landau equation for hard potentials. I. Existence, uniqueness and
smoothness}, Commun. Partial Differ. Equ. \textbf{25} (2000), 179-259.

\bibitem{DV2}
L. Desvillettes, C. Villani, \textit{On the spatially homogeneous Landau equation for hard potentials. II. H-theorem and
applications},  Commun. Partial Differ. Equ. \textbf{25} (2000), 261-298.

\bibitem{DLSS}
R. Duan, S. Liu, S. Sakamoto \& R. Strain, \textit{Global mild solutions of the Landau and non-cutoff Boltzmann equations}, Preprint: https://arxiv.org/pdf/1904.12086.pdf

\bibitem{GIMV}
F. Golse, C. Imbert, C. Mouhot, and A. Vasseur, \textit{Harnack inequality for kinetic Fokker-Planck equations with
rough coefficients and application to the Landau equation},
Ann. Sc. Norm. Super. Pisa Cl. Sci. (5) \textbf{19 }(2019), no. 1, 253-295.
\bibitem{GLX}
L.
 Glangetas, H.-G. Li, C.-J. Xu, \textit{Sharp regularity properties for the non-cutoff spatially homogeneous Boltzmann equation},
Kinet. Relat. Models \textbf{9}(2016), 299-371.

\bibitem{Guo}
Y. Guo, \textit{The Landau equation in a periodic box}, Comm. Math. Phy.
\textbf{231} (2002), 391-434.

\bibitem{Hen-Sn}
Ch. Henderson and S. Snelson, \textit{$C^\infty$ smoothing for weak solutions of the inhomogeneous Landau equations, }
\textit{	Preprint : arXiv:1707.05710}.

\bibitem{landau}
L.D. Landau, Kinetic equation for the case of Coulomb interaction,
\textit{Phys. Zs. Sov. Union}, 10, 154-164 (1936). Translation:The transport equation in the case of
Coulomb interactions, in D. ter  Haar, ed., Collected papers of L. D. Landau, pp. 163-170, Gordon and Breach Science Publishers,
New York-London-Paris 1967.

\bibitem{LMPX2}
N. Lerner, Y. Morimoto, K. Pravda-Starov, C.-J. Xu, \textit{
Phase space analysis and functional calculus for the linearized Landau and Boltzmann operators}.
Kinet. Relat. Models 6(2013), no.3, 625-648


\bibitem{LMPX4}
N. Lerner, Y. Morimoto, K. Pravda-Starov, C.-J. Xu, \textit{Gelfand-Shilov smoothing effect for the spatially inhomogeneous non-cutoff Kac equation},  J.  Funct. Anal. 269 (2015), no. 2, 459-535

\bibitem{pl2} P.L.~Lions,
On Boltzmann and Landau equations,
\textit{Philosophical Transactions: Physical Sciences and Engineering,}
{\bf 346}, 1679,
Mathematics of Non-Linear Systems, 191--204,
Royal Society.


\bibitem{MPX}
Y. Morimoto, K. Pravda-Starov, C.-J. Xu, \textit{A remark on the ultra-analytic smoothing properties of the spatially homogeneous Landau equations}, Kinet. trlat. Models, 6(2013), 715-727.

\bibitem{MU}
Y.~Morimoto,~S. Ukai, \textit{Gevrey smoothing effect of solutions for spatially homogeneous nonlinear Boltzmann equation without angular cutoff,}~J. Pseudo-Differ. Oper., Appl. 1 (2010), no. 1, 139-159.

\bibitem{MUXY-DCDS}
Y. Morimoto, S. Ukai, C.-J. Xu, T. Yang, \textit{Regularity of solutions to the spatially homogeneous Boltzmann equation without angular cutoff}, Discrete Contin. Dyn. Syst. A, 24 (2009), 187-212

\bibitem{MWY}
Y. Morimoto, S. Wang, T. Yang, \textit{Measure valued solutions to the spatially homogeneous Boltzmann equation without angular cutoff},
J. Stat. Phys., \textbf{165} (2016), 866-906.
\bibitem{MX} Y. Morimoto, C.-J. Xu, \textit{Ultra-analytic effect of Cauchy problem for a class of kinetic equations}, J. Differential Equations, 247 (2009), 596-617

\bibitem{SG}
R. Strain, Y. Guo, \textit{ Exponential decay for soft potentials near Maxwellian}, Arch. Ration. Mech. Anal. \textbf{187} (2008),  287-339.

\bibitem{Y-Yu}
T. Yang, H. Yu,
\textit{Optimal convergence rates of the Landau equation with external forcing in the whole space}, Acta Math. Sci., \textbf{29B} (2009), 1035-1062.
\bibitem{villani1}
C. Villani, \textit{On a new class of week solutions to the homogeneous Boltzmann and Landau equations},
Arch. Rat. Mech. anal., \textbf{143} (1998), 273-307.

\bibitem{villani2}
C. Villani, \textit{A review of mathematical topics in collisional kinetic theory}, Handbook of Mathematical Fluid Dynamics, vol. I, North-Holland, Amsterdam (2002) 71-305

\end{thebibliography}
\end{document}